\documentclass[11pt]{amsart}

\usepackage{mathtools}
\usepackage{amsmath, amsthm, amssymb, amsfonts, enumerate}
\usepackage{enumitem}
\usepackage{dsfont}
\usepackage{color}
\usepackage{geometry}
\usepackage{todonotes}
\usepackage{graphicx}
\usepackage{caption}
\usepackage{float}
\usepackage{soul}

\usepackage[colorlinks=true,linkcolor=blue,urlcolor=blue, citecolor=blue]{hyperref}

\mathtoolsset{showonlyrefs}

\def \cA{\mathcal{A}}

\def \cC{\mathcal{C}}

\def \cF{\mathcal{F}}
\def \cG{\mathcal{G}}
\def \cH{\mathcal{H}}
\def \cI{\mathcal{I}}
\def \cJ{\mathcal{J}}
\def \cK{\mathcal{K}}
\def \cL{\mathcal{L}}
\def \cM{\mathcal{M}}

\def \cO{\mathcal{O}}

\def \cQ{\mathcal{Q}}
\def \cR{\mathcal{R}}
\def \cS{\mathcal{S}}
\def \cU{\mathcal{U}}
\def \cT{\mathcal{T}}
\def \P{\mathsf P}

\def \E{\mathsf E}

\def \N{\mathbb{N}}
\def \R{\mathbb{R}}
\def \F{\mathbb F}

\def \ud{\mathrm{d}}
\def \e{\mathrm{e}}

\newcommand{\eps}{\varepsilon}

\newtheorem{theorem}{Theorem}[section]
\newtheorem{lemma}[theorem]{Lemma}
\newtheorem{corollary}[theorem]{Corollary}
\newtheorem{proposition}[theorem]{Proposition}

\newtheorem{definition}[theorem]{Definition}
\newtheorem{remark}[theorem]{Remark}
\newtheorem{assumption}[theorem]{Assumption}

\theoremstyle{definition}

\DeclareMathOperator{\sign}{sign}

\makeatletter
\@namedef{subjclassname@2020}{\textup{2020}Mathematics Subject Classification}
\makeatother

\title[Saddle point of a stopper vs.\ singular-controller game]{On the saddle point of a zero-sum\\ stopper vs.\ singular-controller game}
\author[Bovo]{Andrea Bovo}
\author[De Angelis]{Tiziano De Angelis}
\subjclass[2020]{91A05, 91A15, 60G40, 93E20, 49J40, 35R35, 60J60, 60J55}
\keywords{zero-sum stochastic games, optimal stopping, singular control, saddle point, free boundary problems, Skorokhod reflection, absorbed and controlled diffusions}
\address{A.\ Bovo: School of Management and Economics, Dept.\ ESOMAS, University of Torino, Corso Unione Sovietica, 218 Bis, 10134, Torino, Italy.}
\email{\href{mailto:andrea.bovo@unito.it}{andrea.bovo@unito.it}}
\address{T.\ De Angelis: School of Management and Economics, Dept.\ ESOMAS, University of Torino, Corso Unione Sovietica, 218 Bis, 10134, Torino, Italy; Collegio Carlo Alberto, Piazza Arbarello 8, 10122, Torino, Italy.}
\email{\href{mailto:tiziano.deangelis@unito.it}{tiziano.deangelis@unito.it}}

\date{\today}
\numberwithin{equation}{section}

\begin{document}

\begin{abstract}
We construct a saddle point in a class of zero-sum games between a stopper and a singular-controller. The underlying dynamics is a one-dimensional, time-homogeneous, singularly controlled diffusion taking values either on $\R$ or on $[0,\infty)$. The games are set on a finite-time horizon, thus leading to analytical problems in the form of parabolic variational inequalities with gradient and obstacle constraints. 

The saddle point is characterised in terms of two moving boundaries: an optimal stopping boundary and an optimal control boundary. These boundaries allow us to construct an optimal stopping time for the stopper and an optimal control for the singular-controller. Our method relies on a new link between the value function of the game and the value function of an auxiliary optimal stopping problem with absorption. We show that the smooth-fit condition at the stopper's optimal boundary (in the game), translates into an absorption condition in the auxiliary problem. This is somewhat in contrast with results obtained in problems of singular control with absorption and it highlights the key role of smooth-fit in this context. 
\end{abstract}

\maketitle
\section{Introduction}

Zero-sum stochastic games between a singular-controller and a stopper appeared relatively recently in work by Hernandez-Hernandez and co-authors (\cite{hernandez2015zero}, \cite{hernandez2015zsgsingular}), in diffusive settings. While various economic applications are discussed in the introductions of those papers, from a purely mathematical perspective these games are natural extensions of two important branches of the stochastic control literature. On the one hand, they extend controller-stopper games with classical controls (i.e., controls with bounded velocity) to setups with singular controls (i.e., controls with increments that can be singular with respect to the Lebesgue measure in time). On the other hand, they generalise the classical, single-agent, irreversible (partially reversible) investment problem with singular controls to games with stopping.

Controller-stopper games with classical controls have attracted significant attention, starting with seminal work in discrete-time by Maitra and Sudderth \cite{maitra1996gambler}, later extended to continuous-time setting by Karatzas and Sudderth \cite{karatzas2001controller}. Existence of a value for the game and various characterisations of saddle points have been obtained with different methods spurring the development of analytical techniques (see, e.g., Bensoussan and Friedman \cite{bensoussan1974nonlinear} and, for nonzero-sum games, \cite{bensoussan1977nonzero}), dynamic programming and viscosity theory (e.g., Bayraktar and Huang \cite{bayraktar2013controller}), martingale methods (e.g., Karatzas and Zamfirescu \cite{karatzas2008martingale}) and BSDEs (e.g., Hamadene \cite{hamadene2006stochastic}). Recently, an example of nonzero-sum controller-stopper game with partial and asymmetric information was solved explicitly by Ekstr\"om et al.\ \cite{ekstrom2020detect}.

The literature on (single-agent) irreversible investment is extremely vast and a complete overview is outside the scope of this introduction. Some historical remarks can be found at the end of Chapter VIII in \cite{fleming2006controlled}. Early analytical results for problems of singular control are due to Evans \cite{evans1979second}, Chow, Menaldi and Robin \cite{chow1985additive}, and Menaldi and Taksar \cite{menaldi1989optimal}. Between the 1980s and early 2000s significant contributions to the theory were made by Haussmann and co-authors (e.g., \cite{chiarolla2000inflation}, \cite{chiarolla2009irreversible}, \cite{haussmann1995singular}), Karatzas and co-authors (e.g., \cite{karoui1991new}, \cite{karatzas1984bridge1}, \cite{karatzas1985bridge2}), Soner and Shreve (e.g., \cite{soner1989regularity}, \cite{soner1991free}), among others. This strand of the literature is still very much alive to these days. 

For singular-controller vs.\ stopper (zero-sum) games, the early work of Hernandez-Hernandez and co-authors \cite{hernandez2015zero,hernandez2015zsgsingular} covers one-dimensional diffusive settings with infinite-time horizon. Optimal strategies for both players (and saddle points) are obtained in closed form in terms of {\em stopping boundaries} and {\em control boundaries}. Due to the one-dimensional state-space, such boundaries are points on the real line and they can be determined thanks to an educated {\em a priori} guess about the structure of the players' optimal strategies. Ekstr\"om et al. \cite{ekstrom2023finetti} also considered a nonzero-sum singular-controller vs.\ stopper game, motivated by the classical dividend problem with competition and partial information. They construct an explicit equilibrium using, once again, an educated {\em a priori} guess about the structure of the players' optimal (mixed) strategies. 

The {\em guess-and-verify} approach becomes unfeasible in essentially all multi-dimensional situations as, e.g., two-dimensional diffusive or finite-time horizon games. Motivated by the interest for a more flexible, multi-dimensional framework, in recent work with our co-authors \cite{bovo2023c,bovo2022variational,bovo2023b,bovo2023} we have been developing a systematic approach to the study of zero-sum singular-controller vs.\ stopper games. In our first paper, \cite{bovo2022variational}, we obtained existence of the value of the game as solution in a suitable Sobolev space of a system of variational inequalities with obstacle and gradient constraints on $[0,T]\times\R^d$. We also proved existence of an optimal stopping time for the stopper, which is expressed (in an implicit way) in terms of the game's value evaluated along trajectories of the controlled dynamics. Those results are extended to problems with state-space equal to $[0,T]\times[0,\infty)$ in \cite{bovo2023c}. In that case, additional boundary conditions at zero need to be specified and crucial estimates from \cite{bovo2022variational} need to be derived from scratch with probabilistic methods. In \cite{bovo2023b} and \cite{bovo2023} we extend the existence of the value to two ``{\em degenerate}'' situations: in \cite{bovo2023} the control acts only in a subset of the directions of the state-dynamics and, in \cite{bovo2023b}, the controlled state-coordinates have totally degenerate dynamics (i.e., non-diffusive). In those cases we lose the solvability of the variational inequality in the Sobolev class but we maintain the existence of an optimal stopping time. So far we have been lacking any information about the existence and structure of an optimal control and, therefore, of a saddle point. Moreover, the characterisation of the optimal stopping time is rather abstract, because it involves knowledge of the value function, which is generally difficult to calculate.

In this paper we construct a saddle point (an {\em equilibrium}) in a class of games on finite-time horizon. The underlying controlled dynamics is one-dimensional and time-homogeneous, taking values in either $\R$ or $[0,\infty)$. We provide a characterisation of the players' optimal strategies in terms of two time-dependent optimal boundaries which split the state space into three regions: a {\em continuation/inaction} set, where neither player should act, a {\em stopping} set, $\cS$, where it is optimal for the stopper to end the game, and an {\em action} set, $\cM$, where the singular-controller must act. Moreover, we are able to {\em quantify} the amount of control exerted at equilibrium: the optimal control is constructed via Skorokhod reflection of the controlled dynamics at the boundary of the action set. 

The key steps of our approach are as follows. Starting from the analytical characterisation of the value function $v$ obtained in \cite{bovo2023c,bovo2022variational} we are able to perform an additional probabilistic analysis of the game's value that yields finer properties as, e.g., spatial convexity and monotonicity of $v$, and existence of a time-dependent, non-decreasing optimal stopping boundary. From that we are able to obtain a variational characterisation (in the sense of distributions) for the spatial derivative $v_x$ of the game's value. Then we establish a new probabilistic representation of $v_x$ as the value function of an auxiliary optimal stopping problem (not a game) with absorption, on a suitably chosen underlying dynamics. Thanks to such probabilistic result we can characterise the boundary of the action set $\cM$ as a time-dependent, non-decreasing boundary. That is needed to justify the existence of an optimal control (via its direct construction). Along the way we produce auxiliary results of independent interest as, for example, continuity of the optimal boundaries. We notice that when we work with controlled processes living on $[0,\infty)$ we also encounter fine technical difficulties, related to the fact that the class of admissible controls depends on the initial position of the controlled process (see, e.g., Section \ref{sec:actionstopping}).

Besides the construction of a saddle point, from a mathematical point of view we contribute a new link between our class of zero-sum games and optimal stopping problems with absorption. Such link shows somewhat unexpected features which we now illustrate. 
It is optimal for the stopper to end the game when the controlled diffusion drops below a boundary $t\mapsto a(t)$, for any admissible control of the controller's choosing. Then, from the point of view of the controller the problem can be reduced to a singular control problem with absorption of the controlled diffusion at the boundary $a$. It then appears that the controller's problem belongs in a broad sense to the same class as the celebrated {\em dividend problem} (see, e.g., Radner and Shepp \cite{radner1996risk} and Jeanblanc and Shiryaev \cite{jeanblanc1995optimization}) but with a more complex structure of the cost function, with finite-time horizon and with absorption along a moving bounday, rather than a constant one. Following De Angelis and Ekstr\"om \cite{de2017dividend}, it would then be tempting to argue that $v_x$ should be the value of an optimal stopping problem involving an uncontrolled reflecting diffusion $(Z_t)_{t\ge 0}$, with reflection at the boundary $a$, and with a gain function featuring an exponential of the local time of $Z$ at the boundary $a$. As it turns out this is not the case here. Instead, the underlying diffusion in the auxiliary problem is not reflected but it is absorbed at the boundary $a$ and no local time appears. We explain this phenomenon by noticing that in the classical dividend problem the absorption occurs at an exogenously specified boundary. Instead, in our game, the stopping/absorption boundary is determined endogenously as the optimal stopping boundary for the stopper. Such optimality induces smooth-fit, i.e., continuity of $(t,x)\mapsto v_x(t,x)$ at the boundary $a$. The latter smoothness is lacking in the absorption boundary of the dividend problem, thus leading to different boundary conditions for the spatial derivative of the corresponding value function (cf.\ \cite[Eq.\ (7)]{de2017dividend}).

The paper is organised as follows. In Section \ref{sec:setting} we set up the game, make standing assumptions and recall useful results from \cite{bovo2023c,bovo2022variational}. In Section \ref{sec:actionstopping} we prove existence of an optimal stopping boundary along with its monotonicity and right-continuity. We also prove existence of a (candidate) optimal control boundary. In Section \ref{sec:aux} we obtain the probabilistic representation of the derivative $v_x$ of the game's value. We analyse the auxiliary optimal stopping problem and deduce monotonicity and left-continuity of the (candidate) control boundary. In Section \ref{sec:saddle} we construct the saddle point ({\em equilibrium}) in the game, thus also confirming optimality of the candidate control boundary. In Section \ref{sec:further} we show continuity of the optimal boundaries under some mild additional assumptions on the problem data and, finally, in the Appendix we collect some important technical results.

\section{Problem setting}\label{sec:setting}

Let $(\Omega,\cF,\P)$ be a probability space equipped with a filtration $\F=(\cF_t)_{t\geq0}$ and let $(W_t)_{t\geq0}$ be a one-dimensional Brownian motion adapted to the filtration. Fix a time horizon $T\in(0,\infty)$. The uncontrolled stochastic dynamics $(X^0_t)_{t\in[0,T]}$ is the unique (non-explosive) $\F$-adapted solution of a one-dimensional stochastic differential equation (SDE):
\begin{align}\label{eq:SDE0}
X_t^{0}=x+\int_0^t\mu(X_s^{0}) \ud s+ \int_0^t\sigma(X_s^{0})\ud W_s,\quad t\in[0,T].
\end{align}
The process $X^0$ evolves in an open unbounded domain $\cO$. In particular, we assume that either $\cO=\R$ or $\cO=(0,\infty)$. With a slight abuse of notation we use $\overline\cO$ to indicate $[0,\infty)$ when $\cO=(0,\infty)$ but with the convention $\overline \cO=\R$ when $\cO=\R$. Precise assumptions on the functions $\mu:\overline{\cO}\to\R$ and $\sigma:\overline\cO\to[0,\infty)$ are stated later in Assumption \ref{ass:1a}. The notations $[\![\rho,\tau]\!]$ and $(\!(\rho,\tau)\!)$ are used throughout the paper to indicate, respectively, open and closed random-time intervals for any two random times $\rho\le \tau$, $\P$-a.s.

Admissible controls will be drawn from the class of finite variation c\`adl\`ag processes $\nu$ with Jordan decomposition $\nu=\nu^+-\nu^-$. More precisely, for $(t,x)\in[0,T]\times\cO$ the class of admissible controls is defined as
\begin{align*}
\cA_{t,x}\coloneqq \left\{\nu \left|
\begin{aligned}\,&\,\nu_s=\nu_s^+-\nu_s^-\text{ for all $s\geq 0$ where $(\nu_s^+)_{s\geq0}$ and $(\nu_t^-)_{t\geq 0}$ are $\F$-adapted,}\\
\,&\text{real-valued, non-decreasing, c\`adl\`ag, with $\nu_{0-}^\pm=0$ and $\E[|\nu_{T-t}^\pm|^2]<\infty$;}\\
\,&\text{If $\cO=(0,\infty)$, then $X^\nu_{\tau_\cO}=0$, $\P$-a.s.\ on $\{\tau_\cO<T-t\}$;}
\end{aligned} 
\right. \right\}\!,
\end{align*}
where the controlled dynamics is defined as follows: letting $\tau_\cO\coloneqq\inf\{s\ge 0:X^\nu_s\notin\cO\}\wedge(T-t)$, the process $(X_s^\nu)_{s\in[0,T-t]}$ is the unique $\F$-adapted solution of
\begin{align}\label{eq:SDE}
\begin{aligned}
X_s^{\nu}&=x+\int_0^s\mu(X_v^{\nu}) \ud v+ \int_0^s\sigma(X_v^{\nu})\ud W_v+\nu_s,\quad s\in[\![0,\tau_\cO]\!],\\
\end{aligned}
\end{align}
with $X_s^\nu=0$, for $s\in(\!(\tau_\cO,T-t]\!]$ (with the convention $(a,a]=\varnothing$).
Notice that $X^\nu_{0-}=x\in\cO$ and sometimes we will use the notation $X^{\nu;x}$ to indicate the starting point of the process $X^{\nu}$. When $\cO=(0,\infty)$, admissible controls depend on $x$ via the exit time $\tau_\cO=\tau_{\cO}(\nu;t,x)$, which depends on the choice of control $\nu$ and on the initial position $X^\nu_{0-}=x$. The requirement $X^\nu_{\tau_\cO}=0$, $\P$-a.s.\ on $\{\tau_\cO<T-t\}$ is natural in applications (e.g., in the dividend problem \cite{de2017dividend}) in order to guarantee that a jump in the control cannot exceed the available resources (i.e., $X^\nu_{\tau_\cO-}+\Delta\nu_{\tau_\cO}=0$ on the event $\{\tau_\cO<T-t\}$). Moreover, although the uncontrolled dynamics evolves in $(0,\infty)$, the controlled dynamics $(X^\nu_{s\wedge\tau_\cO})_{s\ge 0}$ should be considered on $[0,\infty)$.
 Finally, the class of admissible stopping times is given by
\begin{align*}
\cT_t\coloneqq\{\tau:\tau\text{ is $\F$-stopping time with $\tau\in[0,T-t]$, $\P$-a.s.}\}.
\end{align*}

Fix $(t,x)\in[0,T]\times \cO$ and an admissible pair $(\tau,\nu)\in\cT_t\times\cA_{t,x}$. To account for the dependence of $X^\nu$ on its starting point $X^\nu_{0-}=x$ we write $\P_x(\,\cdot\,)\coloneqq\P(\,\cdot\,|X^\nu_{0-}=x)$ and $\E_x[\,\cdot\,]\coloneqq\E[\,\cdot\,|X^\nu_{0-}=x]$. We consider an expected payoff of the form 
\begin{align}\label{eq:payoff}
\quad\cJ_{t,x}(\tau,\nu)\coloneqq\E_x\Big[\e^{-r(\tau\wedge\tau_\cO)}g(t+\tau\wedge\tau_\cO)+\int_{0}^{\tau\wedge\tau_\cO} \!\e^{-rs}h(t+s,X_s^{\nu})\,\ud s+\int_{[0,\tau\wedge\tau_\cO]}\!\e^{-rs}\alpha_{0}\,\ud |\nu|_s\Big]
\end{align}
where $r\geq0$ and $\alpha_0>0$ are given constants, $g:[0,T]\to \R$, $h:[0,T]\times\overline\cO\to \R$ are suitable functions and $|\nu|_s$ denotes the total variation measure of $\nu$ on the interval $[0,s]$. The function $\cJ_{t,x}$ is the expected payoff of a zero-sum game in which the {\em maximiser} picks a stopping time $\tau$ and the {\em minimiser} picks an admissible control $\nu$. As usual we introduce the upper and lower value of the game, denoted $\overline v$ and $\underline v$, respectively, and defined as 
\[
\overline v(t,x)=\adjustlimits\inf_{\nu\in\cA_{t,x}}\sup_{\tau\in\cT_t}\cJ_{t,x}(\tau,\nu)\quad\text{and}\quad \underline v(t,x)=\adjustlimits\sup_{\tau\in\cT_t}\inf_{\nu\in\cA_{t,x}}\cJ_{t,x}(\tau,\nu).
\] 
In general, $\underline v\le\overline v$ but we will soon put ourselves under assumptions that guarantee existence of a value for the game, i.e., $\underline v=\overline v\eqqcolon v$.
We are interested in the analysis of the optimal strategies of both players in the sense of the next definition.
\begin{definition}\label{def:NE}
Fix $(t,x)\in[0,T]\times\cO$. An admissible pair $(\tau_*,\nu^*)\in\cT_t\times\cA_{t,x}$ is a saddle point for the game with payoff \eqref{eq:payoff}, evaluated at $(t,x)$, if
\begin{align*}
\cJ_{t,x}(\tau,\nu^*)\le\cJ_{t,x}(\tau_*,\nu^*)\le \cJ_{t,x}(\tau_*,\nu),\quad\text{for all $(\tau,\nu)\in\cT_t\times\cA_{t,x}$}.
\end{align*}
\end{definition}

The main assumptions of the paper are summarised here.
\begin{assumption}[{\bf The dynamics}]\label{ass:1a}
The following holds:
\begin{itemize}
\item[ i)] The function $\mu$ has linear growth, it is continuously differentiable and convex. Its derivative $\mu_x$ is locally $\gamma$-H\"older continuous, for some $\gamma\in(0,1)$ and bounded from above; 

\item[ ii)] When $\cO=(0,\infty)$ we additionally assume $\mu(0)=0$ and $\mu_x(0)=\lim_{x\downarrow 0}\mu_x(x)$ finite (hence $\mu$ is Lipschitz);

\item[iii)] The function $\sigma$ is such that $\sigma(x)=\sigma_0$ when $\cO=\R$, and $\sigma(x)=\sigma_1 x$ when $\cO=(0,\infty)$, for some $\sigma_0,\sigma_1\in(0,\infty)$.
\end{itemize} 
\end{assumption}

\begin{assumption}[\bf Game's payoffs]\label{ass:1}
For some $\gamma\in(0,1)$ and $c>0$ the functions $g:[0,T]\to\R$ and $h:[0,T]\times\overline\cO\to\R$ satisfy:
\begin{itemize}
	\item[ i)] The function $g$ is continuously differentiable with derivative $\dot g$ which is $\gamma$-H\"older continuous.  	
	\item[ ii)] The function $h$ is $\gamma$-H\"older continuous on any compact with $|h(t,x)|\le c(1+|x|^2)$; the spatial derivative $h_x$ is $\gamma$-H\"older continuous on any compact and non-negative, with $x\mapsto h_x(t,x)$ non-decreasing, $t\mapsto h_x(t,x)$ non-increasing and $|h_x(t,x)|\le c(1+|x|^2)$.
\end{itemize}
Letting 
\begin{equation}\label{eq:Theta}
\Theta(t,x)\coloneqq \dot g(t)-rg(t)+h(t,x),\quad \text{for $(t,x)\in[0,T]\times\overline\cO$},
\end{equation}
we further assume 
\begin{itemize}
\item[iii)] $t\mapsto \Theta(t,x)$ is non-increasing for all $x\in\overline\cO$ and, when $\cO=(0,\infty)$, $\Theta(0,0)< 0$;
\item[iv)] There is $K_1>0$ such that 
\[
\inf_{(t,x)\in[0,T]\times\cO}\Theta(t,x)= -K_1;
\]
\item[v)] For all $t\in[0,T)$ we have $\sup_{x\in\cO}\Theta(t,x)>0$. 
\end{itemize}
\end{assumption}
Notice that (i) and (iv) actually imply that $h$ is bounded from below by a constant.

\begin{remark}\label{rem:assv}
Condition $(v)$ in Assumption \ref{ass:1} leads to no loss of generality. 
If there is $t_0\in[0,T)$ such that $\Theta(t_0,x)\le 0$ for all $x\in\cO$, then by monotonicity also $\Theta(t,x)\le 0$ for all $(t,x)\in[t_0,T]\times\cO$. It follows by writing $\e^{-r(\tau\wedge\tau_\cO)}g(t+\tau\wedge\tau_\cO)=g(t)+\int_0^{\tau\wedge\tau_\cO}\! \e^{-rs}(\dot g(t+s)-rg(t+s))\,\ud s$ that 
\[
v(t,x)-g(t)\le \sup_{\tau\in\cT_t}\E_x\Big[\int_0^{\tau\wedge\tau_\cO} \e^{-rs}\Theta(t+s,X^0_s)\ud s\Big]\le 0,\quad\text{for all $(t,x)\in[t_0,T]\times\cO$}.
\]
The latter implies that the stopper ends the game immediately in $[t_0,T]\times\cO$. Thus we can replace the time horizon $T$ by $t_0$.
\end{remark}

Under Assumptions \ref{ass:1a} and \ref{ass:1}, we have from \cite[Thm.\ 3.3]{bovo2022variational} (for $\cO=\R$) and from \cite[Thm.\ 2.3]{bovo2023c} (for $\cO=(0,\infty)$) that the function
\begin{align}\label{eq:game}
v(t,x)=\adjustlimits\sup_{\tau\in \cT_t}\inf_{\nu\in\cA_{t,x}}\cJ_{t,x}(\tau,\nu)=\adjustlimits\inf_{\nu\in\cA_{t,x}}\sup_{\tau\in \cT_t}\cJ_{t,x}(\tau,\nu)
\end{align} 
is well-defined, i.e., the order of supremum and infimum is exchangeable. 
\smallskip

From now on we will make use of various function spaces (e.g., H\"older and Sobolev spaces) for functions defined on $[0,T]\times\cO$. As usual the subscript `$\ell oc$' which appears for example in $L^p_{\ell oc}([0,T]\times\cO)$ is used to denote a class of functions belonging to $L^p(K)$ for any compact $K\subseteq[0,T]\times\cO$.
The infinitesimal generator of the uncontrolled process $X^{0}$ reads
\begin{align*}
(\cL\varphi)(y)=\frac{\sigma^2(y)}{2}\varphi_{xx}(y)+\mu (y)\varphi_{x}(y),\quad\text{for any }\varphi\in C^{2}(\R),
\end{align*}
where $\varphi_x$ and $\varphi_{xx}$ denote the first and second order (spatial) derivatives of the function $\varphi$, respectively. 
It is shown in \cite[Thm.\ 3.3]{bovo2022variational} and in \cite[Thm.\ 2.3]{bovo2023c} that $v\in W^{1,2;p}_{\ell oc}([0,T]\times\cO)$ for all $p\in[1,\infty)$, where $W^{1,2;p}_{\ell oc}$ is the Sobolev space of functions with $\partial_t v,v_x,v_{xx}\in L^p_{\ell oc}([0,T]\times\cO)$. Moreover, $v$ is the maximal solution of the system of variational inequalities
\begin{align}\label{eq:varine}
\begin{aligned}
&\max\big\{\min\big\{\partial_tv+\cL v-rv + h,\alpha_{0} - |v_x|\big\},g-v\big\}=0,\quad\text{a.e.\ $[0,T]\times\cO$},\\
&\min\big\{\max\big\{\partial_tv+\cL v-rv + h,g-v\big\},\alpha_{0}-|v_x|\big\}=0,\quad\text{a.e.\ $[0,T]\times\cO$},
\end{aligned}
\end{align}
with terminal condition $v(T,x)=g(T)$ for all $x\in\cO$, with $|v(t,x)|\le c(1+|x|^2)$ for some $c>0$, and with boundary condition $v(t,0)=g(t)$ for $t\in[0,T]$ when $\cO=(0,\infty)$. By Sobolev embedding the functions $v$ and $v_x$ are locally $\gamma$-H\"older continuous with respect to the parabolic distance (\cite[eq.
(E.9)]{fleming2012deterministic} or \cite[Exercise 10.1.14]{krylov2008lectures}), i.e., $v\in C^{0,1;\gamma}_{\ell oc}([0,T]\times \cO) $ for all $\gamma\in(0,1)$, where the space $C^{0,1;\gamma}$ is defined as in \cite[p.\ 46]{lieberman1996second}. It is shown in \cite[Thm.\ 2.3]{bovo2023c} that when $\cO=(0,\infty)$ the function $v$ is continuous up to and including the boundary $[0,T]\times\{0\}$ of $[0,T]\times\cO$. 
Finally, the stopping time $\tau_*=\tau_*(\nu;t,x)$ defined as
\begin{align}\label{eq:tau*}
\tau_*\coloneqq\inf\big\{s\ge0: \min\big[v(t+s,X_s^{\nu;x}),v(t+s,X_{s-}^{\nu;x})\big]\leq g(t+s)\big\}\wedge(T-t)
\end{align}
is optimal for the stopper in the sense that 
\begin{align}\label{eq:opttau}
v(t,x)=\inf_{\nu\in\cA_{t,x}}\cJ_{t,x}(\tau_*(\nu),\nu).
\end{align}

Motivated by the optimality of $\tau_*$ we introduce the so-called {\em continuation set} for the stopper:
\begin{align}\label{eq:C}
\cC\coloneqq\{(t,x)\in[0,T]\times\cO\,:\,v(t,x)>g(t)\}.
\end{align}
Likewise, the gradient constraint in the system of variational inequalities \eqref{eq:varine} motivates the introduction of the so-called {\em inaction set} for the controller:
\begin{align}\label{eq:I}
\cI\coloneqq\{(t,x)\in[0,T]\times\cO\,:\,|v_x(t,x)|<\alpha_0\}.
\end{align}
The complementary sets to $\cC$ and $\cI$ are denoted by $\cS=\cC^c$ ({\em stopping set}) and $\cM=\cI^c$ ({\em action set}). 

It will be shown in the next sections that indeed it is optimal for the controller to only exert control when the state dynamics $(t+s,X^{\nu;x}_s)$ visits $\cM$. Notice that by construction $\cC\cap(\{T\}\times\cO)=\varnothing$ because $v(T,x)=g(T)$ for all $x\in\cO$. Instead, by the same argument, $\cI\cap(\{T\}\times\cO)=\{T\}\times\cO$ because $v_x(T,x)=0$ for all $x\in\cO$. 

\begin{remark}\label{rem:ass}
In addition to the assumptions from \cite{bovo2023c,bovo2022variational}, in our Assumption \ref{ass:1a} we specify assumptions on the coefficients of the SDE that will be used to show convexity in the space variable of the value function $v$. 
In particular, Assumption \ref{ass:1a} implies that the boundary of the set $\cO$ is non-attainable in finite time by the uncontrolled dynamics $X^0$ (hence the null control $\nu\equiv 0$ is admissible). That can be seen as follows
\begin{itemize}
\item When $\cO=\R$, linear growth of $\mu$ and $\sigma$ implies $X^0_t\in\cO$ for all $t\in[0,T]$, $\P$-a.s. 
\item When $\cO=(0,\infty)$, linear growth of $\mu$ and $\sigma$ implies $X^0_t<\infty$ for all $t\ge 0$, $\P$-a.s. Convexity of $\mu$ and $\mu(0)=0$ imply $\mu(x)\ge \mu_x(0)x$ for all $x\ge 0$; then, for all $t\ge 0$,
\[
0<x\exp(\sigma_1 W_t+(\mu_x(0)-\sigma^2_1/2)t)\le X^0_t<\infty.
\]
\end{itemize} 

When $\cO=\R$ the square integrability of the admissible controls, jointly with Assumption \ref{ass:1a}, imply also that $X^\nu_t\in\R$ for all $t\in[0,T]$ and any $\nu\in\cA_{0,x}$ (hence $\tau_\cO=T$). Instead, when $\cO=(0,\infty)$ the condition $X^\nu_{\tau_\cO}=0$ on the event $\{\tau_\cO<T\}$ must be part of the admissibility requirement for $\nu$ (i.e., the control cannot make the process jump strictly below zero). Finally, the fact that $\mu(0)=\sigma(0)=0$ when $\cO=(0,\infty)$, means that we can formally extend the controlled dynamics $(X^\nu_s)_{s\in[0,T]}$ beyond time $\tau_\cO$ by assuming for $s\in[\![\tau_\cO,T]\!]$, $\nu_{s}-\nu_{\tau_\cO}=0$ so that $X^\nu_{s}=0$. This is done without further mentioning in the rest of the paper (e.g., in the proof of Lemma \ref{lem:convexity}).
\end{remark}

\begin{remark}\label{rem:benchm}
A benchmark example with $\cO=(0,\infty)$, that combines Assumptions \ref{ass:1a} and \ref{ass:1}, is the following: take $g(t)\equiv 0$, $h(t,x)=\kappa_1 x^2-\kappa_2$ for some $\kappa_1,\kappa_2>0$, $\mu(x)=\mu x$ and $\sigma(x)=\sigma x$ for some constants $\mu\in\R$ and $\sigma>0$. Several other forms of the drift and of the payoff functions are of course possible.
\end{remark}

\section{Action and stopping boundaries}\label{sec:actionstopping}

Here we perform an analysis of the game's value function which leads to an initial characterisation of the sets $\cC$ and $\cI$. 
We will establish spatial monotonicity of $v$ and $v_x$ which then yields existence of two time-dependent boundaries for $\cC$ and $\cI$. First we prove $v_x\ge 0$.\begin{proposition}\label{prop:vincr}
The mapping $x\mapsto v(t,x)$ is non-decreasing for every $t\in[0,T]$.
\end{proposition}
\begin{proof}
For $t=T$ the claim is trivial because $v(T,x)=g(T)$. Take $x\geq y$ in $\cO$ and $t\in[0,T)$. For any $\eps>0$ we can find $\nu^\eps\in\cA_{t,x}$ such that $v(t,x)\ge \sup_{\tau\in\cT_t}\cJ_{t,x}(\tau,\nu^\eps)-\eps$. Then we set $\bar \tau\coloneqq\inf\{s\ge 0:X^{\nu^\eps;x}_s\le X^{0;y}_s\}\wedge(T-t)$ and define
\[
\bar\nu^\eps_s\coloneqq 1_{\{\bar\tau\le s\}}\big[-\big(X^{0;y}_{\bar\tau}-X^{\nu^\eps;x}_{\bar\tau}\big)^++\nu^\eps_s-\nu^\eps_{\bar\tau}\big].
\]
Clearly $\bar\nu^\eps$ is a c\`adl\`ag process of finite variation.
By construction, the controlled process $X^{\bar\nu^\eps;y}$ is such that $X^{\bar\nu^\eps;y}_s=X^{0;y}_s$ for $s\in[\![0,\bar \tau)\!)$ and $X^{\bar\nu^\eps;y}_s=X^{\nu^\eps;x}_s$ for $s\in[\![\bar \tau,T\!-\!t]\!]$. Hence, $\bar\nu^\eps\in\cA_{t,y}$.

Notice that $\bar \tau\le \tau_\cO(\nu^\eps;t,x)$ and therefore $\tau_\cO(\nu^\eps;t,x)=\tau_\cO(\bar \nu^\eps;t,y)$. In particular, if the process $X^{\nu^\eps;x}$ goes below $X^{0;y}$ with a jump, then $\bar\nu^\eps$ makes a jump of size $X^{0;y}_{\bar\tau}-X^{\nu^\eps;x}_{\bar\tau}>0$ at time $\bar \tau$. It is crucial to notice that: since $\bar\nu^\eps_s=0$ for $s\in[\![0,\bar\tau)\!)$, $\Delta\nu^\eps_{\bar \tau}=X^{\nu^\eps;x}_{\bar\tau}-X^{\nu^\eps;x}_{\bar\tau-}\le X^{\nu^\eps;x}_{\bar\tau}-X^{0;y}_{\bar\tau-}=\Delta\bar\nu^\eps_{\bar\tau}$ (hence $|\Delta\nu^\eps_{\bar \tau}|\ge|\Delta\bar{\nu}^\eps_{\bar \tau}|$) and $\bar\nu^\eps_s-\bar\nu^\eps_{\bar \tau}=\nu^\eps_s-\nu^\eps_{\bar \tau}$ for $s\in[\![\bar\tau,T\!-\!t]\!]$, then
\[
\int_{[0,\tau\wedge\tau_\cO(\nu^\eps;t,x)]}\!\e^{-rs}\alpha_{0}\,\ud |\nu^\eps|_s\ge \int_{[0,\tau\wedge\tau_\cO(\bar \nu^\eps;t,y)]}\!\e^{-rs}\alpha_{0}\,\ud |\bar\nu^\eps|_s, \quad\text{for any $\tau\in\cT_t$}.
\]
Moreover, using $X^{\nu^\eps;x}_s\ge X^{\bar\nu^\eps;y}_s$ for all $s\in[0,T-t]$ and $h_x\ge 0$ (Assumption \ref{ass:1}(ii)), and setting $\tau_\cO=\tau_\cO(\nu^\eps;t,x)$ and $\bar\tau_\cO=\tau_\cO(\bar \nu^\eps;t,y)$ for simplicity, we have
\begin{align*}
v(t,x)&\ge \sup_{\tau\in\cT_t}\cJ_{t,x}(\tau,\nu^\eps)-\eps\\
&=\sup_{\tau\in\cT_t}\E\Big[\e^{-r(\tau\wedge\tau_\cO)}g(t+\tau\wedge\tau_\cO)+\int_{0}^{\tau\wedge\tau_\cO} \!\e^{-rs}h(t+s,X_s^{\nu^\eps;x})\,\ud s+\int_{[0,\tau\wedge\tau_\cO]}\!\e^{-rs}\alpha_{0}\,\ud |\nu^\eps|_s\Big]-\eps\\
&\ge\sup_{\tau\in\cT_t}\E\Big[\e^{-r(\tau\wedge\bar\tau_\cO)}g(t+\tau\wedge\bar\tau_\cO)+\int_{0}^{\tau\wedge\bar\tau_\cO} \!\e^{-rs}h(t+s,X_s^{\bar\nu^\eps;y})\,\ud s+\int_{[0,\tau\wedge\bar\tau_\cO]}\!\e^{-rs}\alpha_{0}\,\ud |\bar\nu^\eps|_s\Big]-\eps\\
&\ge v(t,y)-\eps.
\end{align*}
Since $\eps>0$ was arbitrary, we conclude.
\end{proof}
\begin{remark}
From the proposition above $v_x\ge 0$. Then, also the form or the inaction set $\cI$ can be slightly simplified and \eqref{eq:I} reduces to
\[
\cI=\{(t,x)\in[0,T]\times\cO:v_x(t,x)<\alpha_0\}.
\] 
\end{remark}

Next we prove convexity of $v(t,\cdot)$. For that we are going to use convexity of the state process with respect to its initial position. Our next lemma is a refinement of a result proven in \cite[Lem.\ 3.1]{boetius1998connection}. In \cite{boetius1998connection} it is assumed that the diffusion coefficient $\sigma$ be independent of the state $X$. Here instead we allow linear dependence and we follow a different argument of proof based on Tanaka's formula. More care is also needed when $\cO=(0,\infty)$, because convex combination of admissible controls may not be admissible in that case.
\begin{lemma}\label{lem:convexity}
Fix $t\in[0,T]$. Let $x_1,x_2\in\cO$, $\lambda\in(0,1)$ and set $x_\lambda=\lambda x_1+(1-\lambda)x_2$. Take $\nu^1\in\cA_{t,x_1}$, $\nu^2\in\cA_{t,x_2}$ and define the c\`adl\`ag process of finite variation $\nu^\lambda\coloneqq\lambda\nu^1+(1-\lambda)\nu^2$. Then 
\[
\lambda X^{\nu^1;x_1}_{s\wedge\theta_\lambda}+(1-\lambda) X^{\nu^2;x_2}_{s\wedge\theta_\lambda}\ge X^{\nu^\lambda;x_\lambda}_{s\wedge\theta_\lambda}\text{ for all $s\in[0,T-t]$,}\quad\text{$\P$-a.s.},
\]
where $\theta_\lambda=\tau_\cO(\nu^\lambda;x_\lambda)=\inf\{s\ge 0:X^{\nu^\lambda;x_\lambda}_{s}\notin\cO\}\wedge(T-t)$.
\end{lemma}
The proof is given in Appendix \ref{app:comparison} for completeness.

\begin{remark}\label{rem:contr}
When $\cO=(0,\infty)$, it may occur that $\nu^\lambda\notin\cA_{t,x_\lambda}$. Indeed, with positive probability it may be $X^{\nu^\lambda;x_\lambda}_{\theta_\lambda-}+\Delta\nu^\lambda_{\theta_\lambda}<0$ on the event $\{\theta_\lambda<T-t\}$, if $|\Delta\nu^1_{\theta_\lambda}|+|\Delta\nu^2_{\theta_\lambda}|>0$ is sufficiently large. However, the dynamics for $(X^{\nu^\lambda;x_\lambda}_{s\wedge\theta_\lambda})_{s\in[0,T-t]}$ remains well-defined if we understand it as the solution of \eqref{eq:SDE} on $[\![0,\theta_\lambda)\!)$ and we take $X^{\nu^\lambda;x_\lambda}_s=X^{\nu^\lambda;x_\lambda}_{\theta_\lambda-}+\Delta\nu^\lambda_{\theta_\lambda}$ for $s\in[\![\theta_\lambda,T-t]\!]$.
\end{remark}

\begin{proposition}\label{prop:vconvex}
The mapping $x\mapsto v(t,x)$ 
is convex for every $t\in[0,T]$.
\end{proposition}
\begin{proof}
For $t=T$ the claim is trivial because $v(T,x)=g(T)$. Let $t\in[0,T)$ and $x_1,x_2\in\cO$. Let also $\lambda\in(0,1)$ and denote $x_\lambda=\lambda x_1+(1-\lambda)x_2$. We only provide the proof in the case $\cO=(0,\infty)$ because for $\cO=\R$ the argument is analogous but much easier. This is due to the fact that for $\cO=\R$ the convex combination of admissible controls is also admissible and $\tau_\cO(\nu;x)=T-t$ for any $x\in\cO$ and $\nu\in\cA_{t,x}$.

We fix $\eta>0$ and, for $i=1,2$, we pick $\nu^{i,\eta}\in\cA_{t,x_i}$ such that 
\[
v(t,x_i)\ge \sup_{\tau\in\cT_t}\cJ_{t,x_i}(\tau,\nu^{i,\eta})-\eta.
\]
For notational simplicity we set $X^{i,\eta}_s=X^{\nu^{i,\eta};x_i}_s$, $\tau_i=\tau_\cO(\nu^{i,\eta};t,x_i)$, for $i=1,2$ and $\tau_{1,2}=\tau_1\wedge\tau_2$. With similar notations as in Lemma \ref{lem:convexity} and with the conventions of Remark \ref{rem:contr}, we consider the control process $\hat \nu^\lambda=\lambda\nu^{1,\eta}+(1-\lambda)\nu^{2,\eta}$, the associated process $\hat X^\lambda=X^{\hat \nu^\lambda;x_\lambda}$ and the exit time $\hat \theta_\lambda=\inf\{s\ge 0:\hat X^\lambda_s\notin\cO\}\wedge(T-t)$. Next, we set 
\[
\nu^\lambda_s\coloneqq\hat \nu^\lambda_s 1_{\{s<\hat \theta_\lambda\}}+\big(\hat \nu^\lambda_{\hat \theta_\lambda -}-\hat X^\lambda_{\hat \theta_\lambda -}\big)1_{\{s\ge \hat \theta_\lambda\}},
\]
with the associated controlled dynamics $X^\lambda=X^{\nu^\lambda;x_\lambda}$ with $\theta_\lambda=\inf\{s\ge 0:X^\lambda_s\notin\cO\}\wedge(T-t)$. By construction the process $\hat \nu^\lambda$ is c\`adl\`ag and non-decreasing. Moreover, $\hat X^\lambda_s= X^\lambda_s$ for $s\in[\![0,\hat\theta_\lambda)\!)$, $\hat \theta_\lambda=\theta_\lambda$ and $X^\lambda_s=0$ for $s\in[\![\theta_\lambda,T-t]\!]$, because $\nu^\lambda_s-\nu^\lambda_{\theta_\lambda}=0$ for $s\in[\![\theta_\lambda,T-t]\!]$. Therefore $\nu^\lambda\in\cA_{t,x_\lambda}$ and $\theta_\lambda=\tau_\cO(\nu^\lambda;t,x_\lambda)$. In this setup $\tau_1\vee\tau_2\ge \theta_\lambda$ because $X^\lambda_s\le \lambda X^{1,\eta}_s+(1-\lambda)X^{2,\eta}_s$ for $s\in[\![0,\theta_\lambda]\!]$ by Lemma \ref{lem:convexity}.

By definition of admissible controls $\nu^{i,\eta}_s-\nu^{i,\eta}_{\tau_i}=0$ for $s\in[\![\tau_i,T-t]\!]$, hence $X^{i,\eta}_s=0$ for $s\in[\![\tau_i,T-t]\!]$ as explained after Assumption \ref{ass:1a} and in Remark \ref{rem:benchm}. Then Lemma \ref{lem:convexity} (cf.\ also Remark \ref{rem:lemconvex}) implies 
\begin{align}\label{eq:tau12}
\begin{aligned}
&X^\lambda_s\le (1-\lambda)X^{2,\eta}_s,\quad &\text{for $s\in[\![\tau_1,\theta_\lambda]\!]$ on the event $\{\tau_{1,2}< \theta_\lambda\}\cap\{\tau_1<\tau_2\}$},\\
&X^\lambda_s\le \lambda X^{1,\eta}_s,\quad &\text{for $s\in[\![\tau_2,\theta_\lambda]\!]$ on the event $\{\tau_{1,2}< \theta_\lambda\}\cap\{\tau_2<\tau_1\}$}.
\end{aligned}
\end{align}
In particular, on the event $\{\tau_1=\tau_2\}$ we have $\theta_\lambda\le \tau_1$ and $X_s^\lambda=0$ for $s\in[\![\tau_1,T-t]\!]$.
Finally, we notice that for any $\tau\in\cT_t$
\begin{align}\label{eq:varnu}
\int_{[0,\tau\wedge\theta_\lambda]}\!\e^{-rs}\alpha_{0}\,\ud |\nu^\lambda|_s\le \int_{[0,\tau\wedge\theta_\lambda]}\!\e^{-rs}\alpha_{0}\,\ud \big(\lambda|\nu^{1,\eta}|_s+(1-\lambda)|\nu^{2,\eta}|_s\big).
\end{align}

Take $\tau_\eta\in\cT_t$ such that $v(t,x_\lambda)\le \inf_{\nu\in\cA_{t,x_\lambda}}\cJ_{t,x_\lambda}(\tau_\eta,\nu)+\eta$. Since $\theta_\lambda\le \tau_1\vee\tau_2$, then 
\begin{align}\label{eq:conv00}
\begin{aligned}
&\lambda v(t,x_1)+(1-\lambda)v(t,x_2)-v(t,x_\lambda)\\
&\geq \lambda\cJ_{t,x_1}(\tau_\eta\wedge \theta_\lambda,\nu^{1,\eta})+(1-\lambda)\cJ_{t,x_2}(\tau_\eta\wedge \theta_\lambda,\nu^{2,\eta})-\cJ_{t,x_\lambda}(\tau_\eta,\nu^{\lambda})-3\eta\\
&\ge\E\Big[\int_0^{\tau_\eta\wedge\theta_\lambda\wedge\tau_{1,2}}\!\e^{-rs}\big(\lambda \Theta(t\!+\!s,X^{1,\eta}_s)+(1-\lambda)\Theta(t\!+\!s,X^{2,\eta}_s)-\Theta(t\!+\!s, X^{\lambda}_s)\big)\,\ud s\Big]\\
&\quad+\E\Big[\mathds{1}_{\{\tau_{1,2}< \tau_\eta\wedge\theta_\lambda\}\cap\{\tau_1< \tau_2\}}\int_{\tau_{1}}^{\tau_\eta\wedge\theta_\lambda}\!\!\!\e^{-rs}\big((1-\lambda)\Theta(t\!+\!s,X^{2,\eta}_s)-\Theta(t\!+\!s, X^{\lambda}_s)\big)\,\ud s\Big]\\
&\quad+\E\Big[\mathds{1}_{\{\tau_{1,2}< \tau_\eta\wedge\theta_\lambda\}\cap\{\tau_1> \tau_2\}}\int_{\tau_{2}}^{\tau_\eta\wedge\theta_\lambda}\!\!\!\e^{-rs}\big(\lambda \Theta(t\!+\!s,X^{1,\eta}_s)-\Theta(t\!+\!s, X^{\lambda}_s)\big)\,\ud s\Big]-3\eta,
\end{aligned}
\end{align}
where we used Dynkin's formula to write the integral in terms of $\Theta$ as in Remark \ref{rem:assv}. 

From the convexity of $x\mapsto \Theta(t,x)$ and using that $X^\lambda_s\le \lambda X^{1,\eta}_s+(1-\lambda)X^{2,\eta}_s$ for $s\in[\![0,\theta_\lambda\wedge\tau_{1,2}]\!]$, we have 
\begin{align}\label{eq:conv01}
\begin{aligned}
&\lambda \Theta(t+s,X^{1,\eta}_s)+(1-\lambda)\Theta(t+s,X^{2,\eta}_s)\\
&\geq \Theta(t+s,\lambda X^{1,\eta}_s+(1-\lambda)X^{2,\eta}_s)\geq \Theta(t+s,X^{\lambda}_s),\quad s\in[\![0,\theta_\lambda\wedge\tau_{1,2}]\!],
\end{aligned}
\end{align}
where the final inequality holds because $\Theta_x\ge 0$. On the event $\{\tau_{1,2}< \tau_\eta\wedge\theta_\lambda\}\cap\{\tau_1< \tau_2\}$ we have 
\begin{align}\label{eq:conv02}
\begin{aligned}
&\int_{\tau_{1}}^{\tau_\eta\wedge\theta_\lambda}\!\!\!\e^{-rs}\big((1-\lambda)\Theta(t\!+\!s,X^{2,\eta}_s)-\Theta(t\!+\!s,X^{\lambda}_s)\big)\,\ud s\\
&\ge \int_{\tau_{1}}^{\tau_\eta\wedge\theta_\lambda}\!\!\!\e^{-rs}\big(\lambda\Theta(t\!+\!s,0)\!+\!(1\!-\!\lambda)\Theta(t\!+\!s,X^{2,\eta}_s)-\Theta(t\!+\!s,X^{\lambda}_s)\big)\,\ud s\\
&\ge \int_{\tau_{1}}^{\tau_\eta\wedge\theta_\lambda}\!\!\!\e^{-rs}\big(\Theta(t\!+\!s,(1\!-\!\lambda)X^{2,\eta}_s)-\Theta(t\!+\!s,X^{\lambda}_s)\big)\,\ud s\ge 0,
\end{aligned}
\end{align}
where the first inequality holds because $\Theta(t+s,0)\le 0$ (cf.\ Assumption \ref{ass:1}(iii)), the second one holds by convexity of $\Theta(t,\,\cdot\,)$ (cf.\ Assumption \ref{ass:1}(ii)) and the final one holds by \eqref{eq:tau12} and $\Theta_x\ge 0$. By analogous arguments, on the event $\{\tau_{1,2}< \tau_\eta\wedge\theta_\lambda\}\cap\{\tau_1> \tau_2\}$ we have 
\begin{align}\label{eq:conv03}
&\int_{\tau_{2}}^{\tau_\eta\wedge\theta_\lambda}\!\!\!\e^{-rs}\big(\lambda\Theta(t\!+\!s,X^{1,\eta}_s)-\Theta(t\!+\!s,X^{\lambda}_s)\big)\,\ud s\ge 0.
\end{align}

Plugging \eqref{eq:conv01}, \eqref{eq:conv02} and \eqref{eq:conv03} into \eqref{eq:conv00} yields
\[
\lambda v(t,x_1)+(1-\lambda)v(t,x_2)-v(t,x_\lambda)\ge -3\eta.
\]
By arbitrariness of $\eta>0$ we deduce convexity of $x\mapsto v(t,x)$, as needed.
\end{proof}

For every $t\in[0,T)$ we set
\begin{align}\label{eq:uppbndry}
a(t)\coloneqq \inf\{y\in\cO:v(t,y)>g(t)\}\quad\text{ and }\quad
b(t)\coloneqq \sup\{y\in\cO:v_x(t,y)<\alpha_{0}\},
\end{align}
with the convention $\inf\varnothing =\sup \cO=+\infty$ and $\sup\varnothing =\inf\cO$. Since $v$ and $v_x$ are continuous and non-decreasing in $x$, then we can parametrise the sets $\cC$ and $\cI$ as 
\[
\cC=\{(t,x)\in[0,T)\times\cO:x>a(t)\}\quad\text{and}\quad\cI=\{(t,x)\in[0,T)\times\cO:x<b(t)\}\cup\big(\{T\}\!\times\!\cO\big).
\]
We refer to $a(\cdot)$ as the {\em stopping boundary} and to $b(\cdot)$ as the {\em action boundary}. Analogously, we parametrise the {\em stopping set} and the {\em action set} as
\[
\cS=\cC^c=\{(t,x)\in[0,T)\times\cO:x\le a(t)\}\cup\big(\{T\}\times\cO\big)
\]
and
\[
\cM=\cI^c=\{(t,x)\in[0,T)\times\cO:x\ge b(t)\},
\]
respectively.
Since $v(t,y)=g(t)$ for $y\leq a(t)$ and $v_x$ is continuous, then $v_x(t,y)=0<\alpha_0$ for $y\le a(t)$ when $\inf\cO<a(t)$. Therefore, $a(t)<b(t)$ for all $t\in[0,T)$ such that $a(t)\in(\inf\cO,\infty)$. Finally, the stopping time $\tau_*$ defined in \eqref{eq:tau*} can be equivalently written as
\begin{align*}
\tau_*=\tau_*(\nu;t,x)=\inf\big\{s\ge 0: \min\big[X^{\nu;x}_s,X^{\nu;x}_{s-}\big]\leq a(t+s)\big\}\wedge(T-t).
\end{align*}

\begin{remark}\label{rem:CI}
For concreteness, we assume throughout that $\cC\cap\cI\neq\varnothing$. The latter always holds for $\cO=(0,\infty)$ (cf.\ Proposition \ref{prop:a-rc}) and for $\cO=\R$ we show sufficient conditions in Proposition \ref{prop:b>0}.
\end{remark}

The use of the two boundaries $a(t)$ and $b(t)$ allows us to rewrite the variational problem \eqref{eq:varine} in a cleaner form as a free boundary problem with two free boundaries. The value function solves in the classical sense the boundary value problem 
\begin{align}\label{eq:PDE}
\begin{aligned}
&\partial_tv(t,x)+\cL v(t,x)-rv(t,x) = -h(t,x),\quad (t,x)\in\cC\cap\cI,\\
&v(t,a(t))=g(t)\quad\text{and}\quad v_x(t,a(t))=0,\quad\quad t\in[0,T),\\
&v_x(t,b(t))=\alpha_0,\qquad\qquad\qquad\qquad\qquad\qquad \!t\in[0,T),
\end{aligned}
\end{align}
with $v(T,x)=g(T)$ and where the boundary conditions at $a$ and $b$ only make sense when $a$ and $b$ are non-trivial, i.e., for $t\in[0,T)$ such that $\inf\cO<a(t)<b(t)<+\infty$. Moreover, the obstacle and gradient constraint continue to hold, i.e., $v(t,x)\ge g(t)$ and $v_x(t,x)\le \alpha_0$ for all $(t,x)\in[0,T]\times\cO$ and the following inequalities hold
\begin{align}\label{eq:VIineq}
\begin{aligned}
&\partial_tv(t,x)+\cL v(t,x)-rv(t,x) \le -h(t,x),\quad\text{for a.e.\ $(t,x)$ such that $x<b(t)$},\\
&\partial_tv(t,x)+\cL v(t,x)-rv(t,x) \ge -h(t,x),\quad\text{for a.e.\ $(t,x)$ such that $x>a(t)$}.
\end{aligned}
\end{align}

Before closing the section we state the main properties of the stopping boundary. 
Recall the function $\Theta$ defined in \eqref{eq:Theta}. By Assumption \ref{ass:1}(ii)-(iii), $\Theta$ is continuous, non-increasing in time and non-decreasing in space. Thus we have that (recall Assumption \ref{ass:1}(v) and Remark \ref{rem:assv})
\begin{align*}
\underline{\Theta}(t)\coloneqq\inf\{y\in\cO:\Theta(t,y)>0\}, 
\end{align*}
is non-decreasing on $[0,T]$ and such that $\Theta(t,x)>0\iff x>\underline\Theta(t)$ for each $t\in[0,T]$. We will use this function to obtain an upper bound for the stopping boundary. 

\begin{proposition}\label{prop:a-rc}
The stopping boundary $t\mapsto a(t)$ is non-decreasing and right-continuous, with $a(t)\le \underline{\Theta}(t)$ for all $t\in[0,T)$. Moreover, when $\cO=(0,\infty)$ we have $a(t)>0$ for all $t\in(0,T]$.
\end{proposition}
\begin{proof}
The proof is divided in steps for the ease of exposition.

\noindent{\bf Step 1}. We first prove that the function $a$ is non-decreasing. For that, it is sufficient to show that $t\mapsto v(t,x)-g(t)$ is non-increasing because that corresponds to
\[
(s,x)\in\cS\implies(t,x)\in\cS,\quad\text{for $0\le s\le t\le T$}.
\]

Fix $x\in\cO$ and $0\le s\leq t\le T$. Since $t\ge s$, then $\cT_t\subseteq \cT_s$ and $\cA_{s,x}\subseteq\cA_{t,x}$. Let $\tau_*$ be an optimal stopping time for $v(t,x)$ and $\nu\in\cA_{s,x}$ be arbitrary. Then, we have 
\begin{align*}
v(t,x)-g(t)&\leq \E_x\Big[\int_0^{\tau_*}\!\e^{-r u}\Theta(t+u,X_{u}^{\nu})\ud u+\int_{[0,\tau_*]}\!\e^{-r u}\alpha_0\ud |\nu|_u\Big]\\
&\leq \E_x\Big[\int_0^{\tau_*}\!\e^{-r u}\Theta(s+u,X_{u}^{\nu})\ud u+\int_{[0,\tau_*]}\!\e^{-r u}\alpha_0\ud |\nu|_u\Big]\\
&\le \sup_{\tau\in\cT_s}\E_x\Big[\int_0^{\tau}\!\e^{-r u}\Theta(s+u,X_{u}^{\nu})\ud u+\int_{[0,\tau]}\!\e^{-r u}\alpha_0\ud |\nu|_u\Big],
\end{align*}
where the second inequality holds because $\Theta(\,\cdot\,,x)$ is non-increasing. Since $\nu\in\cA_{s,x}$ was arbitrary, we conclude that $v(t,x)-g(t)\le v(s,x)-g(s)$, as required. 
\medskip

\noindent{\bf Step 2}. We prove now right-continuity of $a(\cdot)$. Since $\cS$ is closed by continuity of $v$ and $g$, for any {\em decreasing} sequence $(t_n)_{n\in\N}\subset[0,T]$ converging to $t_0\in[0,T)$, we have
\begin{align}\label{eq:rca}
\lim_{n\to\infty}(t_n,a(t_n))=(t_0,a(t_0+))\in\cS\implies a(t_0+)\le a(t_0), 
\end{align}
where $a(t_0+)\coloneqq\lim_{n\to\infty}a(t_n)$ exists because $a$ is monotone. Since $a(t_n)\ge a(t_0)$ for all $n\in\N$, we conclude that $a(t_0+)= a(t_0)$.
\medskip

\noindent{\bf Step 3}.
Next we use monotonicity and right-continuity of $a$ to obtain that $a(t)\leq \underline{\Theta}(t)$ for $t\in[0,T)$. Take an arbitrary $t_0\in[0,T)$. If $a(t_0)=\inf\cO$ the claim is obvious, therefore we focus on $a(t_0)>\inf\cO$. By right-continuity and monotonicity of $a$ we can take an open rectangle $\cR\subset\mathrm{int}(\cS)$ whose left edge lies below the stopping boundary at $t_0$, i.e.
\begin{align}\label{eq:R}
\overline\cR\cap(\{t_0\}\times\cO)\subseteq\{t_0\}\times(\inf\cO,a(t_0)-\eps),
\end{align} 
for some small $\eps>0$. 
Since $v(t,x)=g(t)$ for $(t,x)\in \cR$, then
$(\partial_t+\cL-r)v(t,x)+h(t,x)=\Theta(t,x)$, for $(t,x)\in\cR$.
From \eqref{eq:VIineq}, we know that the left-hand side above must be non-positive, hence $\Theta(t,x)\le 0$ for $(t,x)\in\cR$. By continuity and \eqref{eq:R} we obtain $\Theta(t_0,x)\le 0$ for $x\le a(t_0)-\eps$ and letting $\eps\downarrow 0$ we conclude that $\Theta(t_0,x)\le 0$ for $x\le a(t_0)$. Since $t_0$ was also arbitrary, we deduce $\Theta(t,x)\le 0$ for all $(t,x)\in\cS$. That implies $a(t)\le \underline\Theta(t).$
\medskip

\noindent{\bf Step 4}. 
We conclude by showing that $a(t)>0$ for all $t\in(0,T]$, when $\cO=(0,\infty)$. Arguing by contradiction, let us assume there is $t_0\in(0,T)$ such that $a(t_0)=0$. Then, by monotonicity it must be $a(t)=0$ for $t\in[0,t_0]$. Let $X^{0;x}$ be the uncontrolled process and take an admissible control $\nu$ of the form $\nu_t=0$ for $t\in[0,t_0)$ and $\nu_t=-X^{0;x}_{t_0}$ for $t\in[t_0,T]$. Define $X^{\nu;x}$ as the solution of the SDE \eqref{eq:SDE} controlled by $\nu$. Then $\tau_*(\nu;0,x)= t_0=\tau_\cO(\nu;0,x)$ and we have
\begin{align*}
v(t,x)-g(t)\le \E_x\Big[\int_0^{t_0}\e^{-rs}\Theta(s,X^{\nu}_s)\ud s+\alpha_0\e^{-r t_0}X^\nu_{t_0}\Big].
\end{align*}
Letting $x\downarrow 0$ and applying monotone convergence (recall that $\Theta(s,\,\cdot\,)$ is increasing and bounded from below) we have 
\[
\limsup_{x\downarrow 0}\big(v(t,x)-g(t)\big)\le \int_0^{t_0}\e^{-rs}\Theta(s,0)\ud s<0,
\]
where the final inequality holds by Assumption \ref{ass:1}(iii) and we used $\lim_{x\downarrow 0}\E_x[X^\nu_{t}]=0$ for $t\in[0,t_0]$. The latter holds by standard estimates for SDEs because $\lim_{x\to 0}\big(|\mu(x)|+\sigma(x)\big)=0$ (cf.\ Assumption \ref{ass:1a}). Indeed, for the uncontrolled process $X^0=X^{0;x}$, by H\"older's inequality and Doob's martingale inequality we obtain
\begin{align*}
\E\Big[\sup_{s\in[0,T]}|X_s^0|^2\Big]\le& 3x^2+3\E\Big[\sup_{s\in[0,T]}\Big|\int_0^sb(X_u^0)\ud u\Big|^2+\sup_{s\in[0,T]}\Big|\int_0^s\sigma(X_u^0)\ud W_u\Big|^2\Big]\\
\le&3x^2+3\E\Big[T\int_0^T|b(X_u^0)|^2\ud u+4\int_0^T|\sigma(X_u^0)|^2\ud u\Big]\\
\le&c\Big( x^2+\E\Big[\int_0^T\sup_{v\in[0,u]}|X_v^0|^2\ud u\Big]\Big),
\end{align*}
where $c>0$ is a constant independent of $x$ and, in the final inequality, we used that $\sigma$ is linear and $\mu$ is Lipschitz with $\mu(0)=0$, hence $|\mu(x)|\le L|x|$. By Gronwall's lemma
$\E\big[\sup_{s\in[0,T]}|X_s^0|^2\big]\le c' x^2$ for another constant $c'>0$.
Noticing that $X^{0;x}_s=X^{\nu;x}_s$ for $s<t_0$ and $X^{\nu;x}_s=0$ for $s\ge t_0$, $\P$-a.s., we conclude
$\lim_{x\downarrow 0}\E_x[X^\nu_{t}]=0$, as claimed.
\end{proof}

\section{An auxiliary stopping problem and the action boundary}\label{sec:aux}

In order to determine geometric properties of the action boundary $t\mapsto b(t)$ we must study properties of the function $v_x$. In general, this is a challenging task both from a probabilistic and an analytical point of view. In this section we are able to obtain a probabilistic interpretation of $v_x$ as the value function of an auxiliary optimal stopping problem and to prove that $b(t)$ is the optimal stopping boundary for such new problem. This allows us to obtain monotonicity of $t\mapsto b(t)$.

\subsection{An auxiliary stopping problem}

The existence of the value function $v$ of the game \eqref{eq:game} is proved in \cite{bovo2023c,bovo2022variational} by a penalisation procedure that goes as follows: take $\eps,\delta\in(0,1)$, denote $(c)^+=\max\{0,c\}$ and let $\psi_\varepsilon\in C^2(\R)$ be convex, non-decreasing, with $\psi_\varepsilon(y)=0$ if $y\leq0$ and $\psi_\varepsilon(y)=\frac{y-\varepsilon}{\varepsilon}$ if $y\geq 2\varepsilon$. By \cite[Thm.\ 5.3+Prop.\ 5.4]{bovo2022variational} and \cite[Thm.\ 4.6]{bovo2023c}, there exists a unique solution $v^{\varepsilon,\delta}\in C^{1,2;\gamma}_{\ell oc}([0,T)\times\cO)$ with $v^{\eps,\delta} \in C([0,T]\times\cO)$ and $v^{\eps,\delta}_x\in C^\gamma_{\ell oc}([0,T)\times\cO)$ (for any $\gamma\in(0,1)$), of the following semi-linear PDE: 
\begin{align}\label{eq:PDE_pen}
\begin{cases}
(\partial_t +\cL-r)v^{\varepsilon,\delta}=-h-\frac{1}{\delta}(g-v^{\varepsilon,\delta})^++\psi_\varepsilon(|v^{\varepsilon,\delta}_x|^2-\alpha_{0}^2),&\quad\text{on }[0,T)\times\cO,\\
v^{\varepsilon,\delta}(T,y)=g(T),&\quad y\in\cO,
\end{cases}
\end{align} 
where $v^{\eps,\delta}$ has at most quadratic growth. If $\cO=(0,\infty)$, an additional boundary condition $v^{\eps,\delta}(t,0)=g(t)$ is necessary for $t\in[0,T)$ and $v^{\eps,\delta}$ is continuous up to and including the boundary $[0,T)\times\{0\}$.

It is shown in \cite[Thm.\ 6.1]{bovo2022variational} and in \cite[Sec.\ 5]{bovo2023c} that for any $\gamma\in(0,1)$ and $p \in (1,\infty)$,
\begin{align}\label{eq:prp_pen_prb}
\lim_{\eps,\delta\to 0}v^{\eps,\delta}=v\:\text{in $C^{0,1;\gamma}_{\ell oc}([0,T]\times\cO)$ and, weakly, in $W^{1,2;p}_{\ell oc}([0,T]\times\cO)$}. 
\end{align}
The proof of the above limits uses a bound which will be required later in this paper. It is shown in \cite[Lemma 5.8]{bovo2022variational} and in \cite[Lemma 4.9]{bovo2023c} (using the same notation) that for each compact $\cK\subset[0,T]\times\cO$ there is a constant $M=M(\cK)>0$ such that 
\begin{align}\label{eq:boundpsi}
\big\|\psi_\eps\big(|v^{\eps,\delta}_x|^2-\alpha_0^2\big)\big\|_{C^0(\cK)}\le M,\quad\text{for all $\eps,\delta\in(0,1)$},
\end{align}
where $\|\cdot\|_{C^0(\cK)}$ denotes the supremum norm on the space of continuous functions on $\cK$.

The function $v^{\eps,\delta}$ enjoys a probabilistic representation which involves two classes of admissible controls. For $t\in[0,T]$ and $x\in\cO$, we introduce the class of absolutely continuous controls 
\begin{align}\label{eq:Acirc}
\cA^\circ_{t,x}\coloneqq\{\nu\in\cA_{t,x}: t\mapsto\nu_t\text{ is absolutely continous on }[0,T-t]\}.
\end{align}
Also, we let
\begin{align*}
\cT_t^\delta\coloneqq\left\{w\,:\begin{aligned} &\,(w_s)_{s\ge 0}\text{ is progressively measurable,} \\
&\text{with $0 \leq w_s \leq \tfrac{1}{\delta}$, $\P$-a.s. for all $s \in [0, T-t]$}
\end{aligned}\right\} .
\end{align*}

Given $\xi\in\cA_{t,x}^\circ$, let $X^{\xi}$ be the solution of the SDE 
\begin{align}\label{eq:SDE_pen}
X^{\xi}_s=x+\int_0^s\big[\mu(X_u^\xi)+\dot{\xi}_u\big]\ud u+\int_0^s\sigma(X_u^\xi)\ud W_u,\quad\text{for $s\in[0,T-t]$}.
\end{align}
Then $v^{\varepsilon,\delta}$ admits the following representation (see \cite[Prop.\ 5.4]{bovo2022variational}): setting
\[
H^\eps(t,x;\lambda,\zeta)\coloneqq h(t,x)\!+\!\lambda g(t)\!+\!\sup_{p\in\R}\!\big\{p\,\zeta\!-\!\psi_\varepsilon(p^2\!-\!\alpha_{0}^2)\big\},
\]
and $\tau_\cO=\tau_\cO(\xi;t,x)=\inf\{s\ge 0: X^\xi_s\le 0\}\wedge(T-t)$, we have
\begin{align}\label{eq:pen_prb_rep}
&\quad v^{\varepsilon,\delta}(t,x)=\adjustlimits\sup_{w\in\cT^\delta_t}\inf_{\xi\in\cA^\circ_{t,x}}\E_{x}\Big[\e^{-\int_0^{\tau_\cO}(r+w_s)\ud s}g(t\!+\!\tau_\cO)\!+\!\int_0^{\tau_\cO}\!\!\!\e^{-\int_0^s(r+w_u)\ud u}H^\eps(t\!+\!s,X^{\xi}_s;w_s,\dot \xi_s)\,\ud s\Big],
\end{align}
where the order of supremum and infimum can be swapped without changing the value of the function. 
\begin{lemma}\label{lem:vepsdelta}
There is $c\ge 0$, independent of $(\eps,\delta)$, such that $v^{\eps,\delta}\ge -c$. Moreover, the mapping $x\mapsto v^{\eps,\delta}(t,x)$ is non-decreasing and convex for every $t\in[0,T]$.
\end{lemma}
\begin{proof}
Integration by parts yields
\begin{align*}
\E_{x}\Big[\e^{-\int_0^{\tau_\cO}(r+w_s)\ud s}g(t+\tau_\cO)\Big]=g(t)+\E_{x}\Big[\int_0^{\tau_\cO}\!\e^{-\int_0^s(r+w_u)\ud u}\Big(\dot g(t+s)-(r+w_s)g(t+s)\Big)\!\,\ud s\Big],
\end{align*}
for any $t\in[0,T]$ and $w\in\cT_t^\delta$. Thus, using that $\sup_{p\in\R}\big\{p\,\dot{\xi}_s-\psi_\varepsilon(p^2-\alpha_{0}^2)\big\}\geq 0$ (choose $p=0$) and plugging the equation above into the probabilistic representation \eqref{eq:pen_prb_rep}, we get
\begin{align*}
v^{\varepsilon,\delta}(t,x)\geq 
g(t)+\adjustlimits\sup_{w\in\cT^\delta_t}\inf_{\xi\in\cA^\circ_{t,x}}\E_{x}\Big[&\int_0^{\tau_\cO}\!\e^{-\int_0^s(r+w_u)\ud u}\Theta(t\!+\!s,X^{\xi}_s)\!\,\ud s\Big].
\end{align*}
Using Assumption \ref{ass:1}(iv), $w_s\ge 0$ and continuity of $g$, we have
$v^{\varepsilon,\delta}(t,x)\geq -K_1T+\min_{t\in[0,T]}g(t)$.

Monotonicity and convexity in the spatial variable are obtained repeating arguments as those in the proof of Propositions \ref{prop:vincr} and \ref{prop:vconvex} and using that $x\mapsto H^\eps(t,x;\lambda,\zeta)$ is increasing and convex. We omit further details for brevity. 
\end{proof}

Our next task is to prove that $v_{x}$ satisfies a variational inequality in the sense of distributions. We introduce the operator $\cG$ whose action on a regular function $\Gamma\in C^2(\R)$ is 
\begin{align}\label{eq:G}
(\cG\Gamma)(y)\coloneqq \frac{\sigma^2(y)}{2}\Gamma_{xx}(y)+\Big[\sigma(y)\sigma_x(y)+\mu(y)\Big]\Gamma_x(y)-\Big[r-\mu_x(y)\Big]\Gamma(y).
\end{align}
Let $\cU\subset[0,T]\times\cO$ be an arbitrary set and $\ell\in L^\infty_{\ell oc}(\cU)$ (by our convention $L^\infty_{\ell oc}(\cU)=L^\infty(\cU)$ when $\cU$ is bounded). Following \cite{lamberton2008critical}, we say that a function $f$ is a solution of $(\partial_tf+\cG f) \geq \ell$ on $\cU$ in the sense of distributions if 
\begin{align*}
\int_0^T\!\!\int_{\cO}\!f(t,x) \big[(-\partial_t+\cG^*)\varphi\big](t,x)\,\ud x \ud t\geq \int_0^T\!\!\int_\cO\! \ell(t,x)\varphi(t,x)\,\ud x\,\ud t, 
\end{align*}
for any {\em non-negative} test function $\varphi\in C^\infty_c(\cU)$ and with $\cG^*$ the adjoint of $\cG$. For future reference, we notice that $\cG^*$ acts on $\Gamma\in C^2(\R)$ as 
\begin{align}\label{eq:G*}
(\cG^*\Gamma)(y)\coloneqq \tfrac{\partial^2}{\partial x^2}\Big(\tfrac{\sigma^2(y)}{2}\Gamma(y)\Big)-\tfrac{\partial}{\partial x}\Big(\big(\sigma(y)\sigma_x(y)+\mu(y)\big)\Gamma(y)\Big)-\big(r-\mu_x(y)\big)\Gamma(y).
\end{align}

\begin{theorem}\label{thm:v_xVI}
The function $v_x$ is a (continuous) solution in the sense of the distributions of 
\begin{align}\label{eq:vxPDEd}
\big(\partial_t v_x+\cG v_x\big)(t,x)\geq -h_x(t,x),\quad (t,x)\in\cC,
\end{align}
with $v_x(t,a(t))=0$, for $t\in[0,T)$ such that $a(t)>\inf\cO$, and $v_x(T,x)=0$ for $x\in\cO$. Moreover, the function $v_x$ is a {\em classical solution} of
\begin{align}\label{eq:vxPDE}
\big(\partial_tv_x+\cG v_x\big)(t,x)= -h_x(t,x),\quad (t,x)\in\cC\cap\cI.
\end{align}
\end{theorem}
\begin{proof}
We prove the second claim first. Since $\cC\cap\cI$ is an open set we can take an arbitrary open subset $\cR\subset\cC\cap\cI$ with smooth parabolic boundary. 
Recall that $v$ is classical solution of \eqref{eq:PDE}. Since the coefficients of the infinitesimal operator $\cL$ admit $\gamma$-H\"older derivative in $\cR$, then classical interior regularity results for parabolic PDEs \cite[Thm.\ 3.10]{friedman2008partial} yield the analogue of \eqref{eq:vxPDE} in $\cR$. The result extends to $\cC\cap\cI$ by arbitrariness of $\cR$. 

Now we prove the first claim. The boundary conditions at $\{(t,a(t)),\,t\in[0,T)\}$ and at $\{T\}\times\cO$ are obvious by continuity of $v_x$. Recall $\cC=\{v>g\}$ and for $n\in\N$ let us define 
\begin{align}\label{eq:Cn}
\cC_n=\{(t,x)\in [0,T]\times \cO: v(t,x)>g(t)+\tfrac{1}{n}\}.
\end{align}
The set $\cC_n$ is open and connected\footnote{Actually, an adaptation of Proposition \ref{prop:a-rc} would allow us to prove that there is $t\mapsto a_n(t)$ which is non-decreasing and right-continuous, with $a_n(t)<\infty$ for $t\in[0,T)$ and $a_n(T-)=+\infty$, and such that $\cC_n=\{(t,x):x>a_n(t)\}$. We omit these calculations because we do not need to work with the boundary $a_n$.} because $v$ and $g$ are continuous with $v_x\geq 0$ and $v_t-\dot g\le 0$ (cf.\ step 1 in the proof of Proposition \ref{prop:a-rc}). We also notice that $\cC_n\subset [0,T)\times\cO$, because $v(T,x)=g(T)$.

Take an arbitrary compact set $K\subset \cC$. Then take an open bounded set $U\subseteq[0,T]\times\cO$ such that $K\subset U$. Since $\cC$ is open and $(\cC_n)_{n\in\N}$ is an increasing sequence of open sets converging to $\cC$, then we can find $n_K\in\N$ sufficiently large that $K\subset\cC_n\cap U$ for all $n\ge n_K$. We assume with no loss of generality that from now on $n\ge n_K$.

Recalling $v^{\eps,\delta}$ from \eqref{eq:PDE_pen}, the uniform convergence in \eqref{eq:prp_pen_prb}, guarantees that there exist $\eps_{n,U}>0$ and $\delta_{n,U}>0$ such that 
\begin{align}\label{eq:low}
\inf_{(t,x)\in \overline{\cC_n\cap U}}\big(v^{\eps,\delta}(t,x)-g(t)\big)\geq \tfrac{1}{2n},
\end{align}
for all $\eps\in(0,\eps_{n,U})$ and $\delta\in(0,\delta_{n,U})$.
It follows from \eqref{eq:PDE_pen} that
\begin{align}\label{eq:os0}
\big[(\partial_t+\cL-r)v^{\eps,\delta}\big](t,x)=-h+\psi_\eps(|v_x^{\eps,\delta}(t,x)|^2-\alpha_0^2) \quad\text{for $(t,x)\in\cC_n\cap U$,}
\end{align}
in the classical sense, i.e., with continuous derivatives. Since $v^{\eps,\delta}\in C^{1,2;\gamma}_{\ell oc}([0,T)\times\cO)$ for any $\gamma\in(0,1)$, then 
\begin{align}\label{eq:os1}
-h+\psi_\eps(|v_x^{\eps,\delta}|^2-\alpha_0^2)\in C^{0,1;\gamma}_{\ell oc}([0,T)\times\cO),
\end{align} 
for any fixed $\eps>0$, $\delta>0$.
Moreover, the spatial derivative reads
\begin{align}\label{eq:os2}
-h_x+2v_x^{\eps,\delta}v_{xx}^{\eps,\delta}\dot{\psi_\eps}(|v_x^{\eps,\delta}|^2-\alpha_0^2)\ge -h_x,
\end{align}
where the inequality holds because $\dot\psi_\eps\ge 0$ and $v^{\eps,\delta}(t,\cdot)$ is non-decreasing and convex (Lemma \ref{lem:vepsdelta}).

Thanks to interior regularity results for parabolic PDEs (see \cite[Thm.\ 3.10]{friedman2008partial}) it follows from \eqref{eq:os1} and the regularity of $\mu$ and $\sigma$ (Assumption \ref{ass:1a}) that $v_x^{\eps,\delta}\in C^{1,2;\gamma}(\cC_n\cap U)$. Differentiating \eqref{eq:os0} and using \eqref{eq:os2} we obtain
\begin{align}\label{eq:vxepsdel}
\big[(\partial_t+\cG)v^{\eps,\delta}_x\big](t,x)+h_x(t,x)\geq0 \quad \text{for $(t,x)\in\cC_n\cap U$,}
\end{align}
and any $(\eps,\delta)\in(0,\eps_{n,U})\times(0,\delta_{n,U})$, with $\cG$ as in \eqref{eq:G}. 

Recall that $K\subset\cC_n\cap U$ and let $\varphi^K\in C^\infty_c(K)$ be a non-negative test function. Multiply both sides of \eqref{eq:vxepsdel} by $\varphi^K$ and integrate by parts to obtain
\begin{align*}
\int_0^T\!\int_\cO v^{\eps,\delta}_x(t,x)\big[(-\partial_t\varphi^K+\cG^*\varphi^K)(t,x)\big]\ud x\ud t\ge -\int_0^T\!\int_\cO h_x(t,x)\varphi^K(t,x)\ud x \ud t.
\end{align*} 
Letting $\varepsilon,\delta\to0$, we can use the dominated convergence theorem because $v^{\varepsilon,\delta}_x$ is locally bounded in $[0,T]\times\cO$, uniformly in $(\eps,\delta)$. Since $v^{\varepsilon,\delta}_x\to v_x$ uniformly on compacts (see \eqref{eq:prp_pen_prb}), it follows that
\begin{align}\label{eq:distribution}
\int_0^T\!\int_{\cO}v_x(t,x)\big[(-\partial_t\varphi^K +\cG^*\varphi^K)(t,x)\big]\,\ud x\,\ud t\geq \int_0^T\!\int_{\cO}\!-h_x(t,x)\varphi^K(t,x)\,\ud x\,\ud t.
\end{align}
Since $K\subset\cC\cap U$ was arbitrary and $U\in[0,T]\times\cO$ was also arbitrary, we can conclude that 
$v_x$ solves
$\big(\partial_tv_x+\cG v_x\big)\geq -h_x$ in $\cC$ in the sense of distributions.
\end{proof}

Let us now introduce an auxiliary dynamics which will appear in the probabilistic representation of $v_x$. 
Since $\mu$ is locally Lipschitz and $\sigma$ is affine, denoting $\sigma_x$ the spatial derivative of $\sigma$, the SDE 
\begin{align}\label{eq:SDEY}
 Y_t=y+\int_0^t \big(\mu(Y_s)+\sigma(Y_s)\sigma_x(Y_s)\big)\ud s+ \int_0^t\sigma(Y_s) \ud W_s,\quad t\in[0,T],
\end{align}
admits a unique strong solution. We notice that the end-points of $\cO$ remain unattainable for the process $Y$. This is clear for $\cO=\R$, because the SDE is non-explosive. For $\cO=(0,\infty)$ we additionally notice that $\sigma(y)\sigma_x(y)=\sigma_1^2 y>0$, so the process $Y$ has a larger positive drift than $X^0$. 

We denote
\begin{align}\label{eq:taua}
\tau_a=\tau_a(t,y)\coloneqq \inf\{s\ge 0: Y_{s}^y\leq a(t+s)\}\wedge(T\!-\!t),
\end{align}
the exit time of the process $(t+s,Y_s^y)$ from the set $\cC$. We also define $\lambda(y)\coloneqq r-\mu_x(y)$, which will be the discount factor in \eqref{eq:optstop} below.

We know a priori that $v_x$ is a bounded continuous function, with $0\leq v_x\leq \alpha_{0}$. Moreover, it solves \eqref{eq:vxPDEd} in the sense of distributions and \eqref{eq:vxPDE} in the classical sense. In principle, we would like to apply the verification theorem \cite[Prop.\ 2.5]{lamberton2008critical} but upon close inspection of the proof, it turns out that our state-dependent coefficients cause some technical difficulties. To overcome those, we take a different approach based once again on the penalised problem \eqref{eq:PDE_pen}.

\begin{theorem}\label{thm:prbrprvx}
For $(t,y)\in[0,T]\times \cO$ we have 
\begin{align}\label{eq:optstop}
v_x(t,y)=\inf_{\tau\in\cT_t}\E\Big[\e^{-\int_0^\tau\lambda(Y_u^y)\ud u}\alpha_{0}\mathds{1}_{\{\tau< \tau_a(t,y)\}}+\int_{0}^{\tau\wedge\tau_a(t,y)}\!\e^{-\int_0^s\lambda(Y_u^y)\ud u}h_x(t+s,Y_s^y)\,\ud s\Big].
\end{align}

Moreover, the stopping time $\sigma_*=\sigma_*(t,y)$ defined as
\begin{align}\label{eq:tau*vx}
\sigma_*\coloneqq\inf\{s\ge 0 :v_x(t+s,Y_s^{y})\geq\alpha_{0}\}\wedge(T-t),
\end{align}
is optimal.
\end{theorem}
\begin{proof}
Here we recall the notation from the proof of Theorem \ref{thm:v_xVI}. Fix $(t,y)\in\cC_n$ (cf.\ \eqref{eq:Cn}) and let $U\subset[0,T]\times\cO$ be an open bounded set. Define
\begin{align*}
\rho_n\coloneqq\inf\{s\ge 0:(t+s,Y_s^y)\notin \cC_n\}\quad\text{ and }\quad\rho_U\coloneqq\inf\{s\ge 0 :(t+s,Y_s^y)\notin U\},
\end{align*}
(notice that both $\rho_n$ and $\rho_U$ are in $[0,T-t]$). Set $\bar\rho=\bar{\rho}_{n,U}\coloneqq \rho_{n}\wedge\rho_{U}$. 
Recall from the proof of Theorem \ref{thm:v_xVI} that we can choose $\eps_{n,U}, \delta_{n,U}>0$ such that \eqref{eq:low} holds for $\eps\in(0,\eps_{n,U})$ and $\delta\in(0,\delta_{n,U})$, hence implying also \eqref{eq:os0} and \eqref{eq:vxepsdel}. Moreover, $v^{\eps,\delta}_x\in C^{1,2;\gamma}(\cC_n\cap U)$ for any $\gamma\in(0,1)$. Hence, for any $\tau\in\cT_{t}$, an application of It\^o's formula yields
\begin{align}\label{eq:itovx} 
\begin{aligned}
v^{\eps,\delta}_x(t,y)=&\,\e^{-\int_0^{\bar{\rho}\wedge\tau}\lambda(Y_s^y)\,\ud s}
v^{\eps,\delta}_x(t\!+\!\bar\rho\wedge\tau,Y^y_{\bar\rho\wedge\tau})\!-\!\int_0^{\bar{\rho}\wedge\tau}\!\!\e^{-\int_0^{s}\lambda(Y_u^y)\,\ud u}(\partial_t\!+\!\cG)v^{\eps,\delta}_x(t\!+\!s,Y^y_{s})\,\ud s\\
&-\int_0^{\bar{\rho}\wedge\tau}\!\e^{-\int_0^{s}\lambda(Y_u^y)\,\ud u}v^{\eps,\delta}_{xx}(t+s,Y^y_{s})\sigma(Y^y_{s})\,\ud W_s.
\end{aligned}
\end{align}
Using \eqref{eq:vxepsdel}, since the stochastic integral is a martingale we obtain
\begin{align*}
v^{\eps,\delta}_x(t,y)\leq &\,\E\Big[\e^{-\int_0^{\bar{\rho}\wedge\tau}\lambda(Y_s^y)\,\ud s}v^{\eps,\delta}_x(t+\bar\rho\wedge\tau,Y^y_{\bar\rho\wedge\tau})+\int_0^{\bar{\rho}\wedge\tau}\!\e^{-\int_0^{s}\lambda(Y_u^y)\,\ud u}h_x(t+s,Y^y_{s})\,\ud s\Big].
\end{align*}
 
By uniform convergence \eqref{eq:prp_pen_prb} on compacts, we can pass to the limit as $\eps,\delta\downarrow0$ and we obtain
\begin{align}\label{eq:limvx1}
v_x(t,y)\leq &\,\E\Big[\e^{-\int_0^{\bar{\rho}\wedge\tau}\lambda(Y_s^y)\,\ud s}v_x(t+\bar\rho\wedge\tau,Y^y_{\bar\rho\wedge\tau})+\int_0^{\bar{\rho}\wedge\tau}\!\e^{-\int_0^{s}\lambda(Y_u^y)\,\ud u}h_x(t+s,Y^y_{s})\,\ud s\Big].
\end{align}
Letting $U\uparrow (0,T)\times\cO$ we have $\rho_U\uparrow T-t$, $\P$-a.s.,\ because the process $Y$ does not leave $\cO$ in finite time. Passing to the limit in \eqref{eq:limvx1}, the integral term converges by monotone convergence. Dominated convergence can be applied to the first term because $v_x$ is bounded and non-negative and, using $\lambda(y)=r-\mu_x(y)\ge -\mu_x(y)$, we have
\begin{align*}
0\le \e^{-\int_0^{\bar{\rho}\wedge\tau}\lambda(Y_s^y)\,\ud s}v_x(t+\bar\rho\wedge\tau,Y^y_{\bar\rho\wedge\tau})\leq \alpha_0\e^{\int_0^{\bar{\rho}\wedge\tau}\mu_x(Y_s^y)\,\ud s}\le \alpha_0\e^{T\sup_{y\in\cO}\mu_x(y)},
\end{align*}
where we recall that $\mu_x$ is bounded from above (Assumption \ref{ass:1a}).

In conclusion, letting $U\uparrow(0,T)\times\cO$ in \eqref{eq:limvx1} we obtain
\begin{align}\label{eq:limvx2}
v_x(t,y)\leq &\,\E\Big[\e^{-\int_0^{\rho_n\wedge\tau}\lambda(Y_s^y)\,\ud s}v_x(t\!+\!\rho_n\wedge\tau,Y^y_{\rho_n\wedge\tau})\!+\!\int_0^{\rho_n\wedge\tau}\!\!\e^{-\int_0^{s}\lambda(Y_u^y)\,\ud u}h_x(t\!+\!s,Y^y_{s})\,\ud s\Big].
\end{align}
Next we want to pass to the limit as $n\to\infty$. 
The sequence $(\rho_n)_{n\in\N}$ is non-decreasing with $\rho_n\le \tau_a$ (recall \eqref{eq:taua}) and its limit $\rho_\infty=\lim_{n\to\infty}\rho_n$ is well-defined with $\rho_\infty\le \tau_a$. By definition of $\rho_n$ it is not hard to show that $\rho_\infty=\inf\{s\ge 0:v(t+s,Y^y_s)\le g(t+s)\}\wedge(T-t)=\tau_a$, thanks to continuity of $s\mapsto v(t+s,Y^y_s)- g(t+s)$.
Letting $n\to\infty$ in \eqref{eq:limvx2} and again using dominated convergence and monotone convergence we obtain
\begin{align*}
v_x(t,y)\leq \E\Big[\e^{-\int_0^{\tau_a\wedge\tau}\lambda(Y_s^y)\,\ud s}v_x(t\!+\!\tau_a\wedge\tau,Y^y_{\tau_a\wedge\tau})1_{\{\tau\wedge\tau_a<T-t\}}\!+\!\int_0^{\tau_a\wedge\tau}\!\!\e^{-\int_0^{s}\lambda(Y_u^y)\,\ud u}h_x(t\!+\!s,Y^y_{s})\,\ud s\Big],
\end{align*}
where the indicator function appears because $v_x(T,\,\cdot\,)=0$.

On $\{\tau_a\le \tau\}\cap\{\tau_a<T-t\}$ we have $v_x(t+\tau_a,Y_{\tau_a}^y)=v_x(t+\tau_a,a(t+\tau_a))=0$. Then, recalling $v_x\le \alpha_0$ and $\tau_a\le T\!-\!t$, $\P$-a.s., we obtain
\begin{align}\label{eq:limvx3}
v_x(t,y)\leq \E\Big[\e^{-\int_0^{\tau}\lambda(Y_s^y)\,\ud s}\alpha_0\mathds{1}_{\{\tau<\tau_a(t,y)\}}+\int_0^{\tau_a(t,y)\wedge\tau}\!\e^{-\int_0^{s}\lambda(Y_u^y)\,\ud u}h_x(t+s,Y^y_{s})\,\ud s\Big],
\end{align}
where we restored the notation $\tau_a=\tau_a(t,y)$. By arbitrariness of $\tau\in\cT_t$, we have
\begin{align}\label{eq:limvx3a}
v_x(t,y)\leq \inf_{\tau\in\cT_t}\E\Big[\e^{-\int_0^{\tau}\lambda(Y_s^y)\,\ud s}\alpha_0\mathds{1}_{\{\tau<\tau_a(t,y)\}}+\int_0^{\tau_a(t,y)\wedge\tau}\!\e^{-\int_0^{s}\lambda(Y_u^y)\,\ud u}h_x(t+s,Y^y_{s})\,\ud s\Big].
\end{align}

Let us now show that the inequality is actually an equality. For $n\in\N$ let 
\[
\cI_n\coloneqq\{(t,x)\in[0,T]\times\cO: |v_x(t,x)|< \alpha_0-\tfrac1n\}.
\]
With no loss of generality we assume $\eps_{n,U}>0$ and $\delta_{n,U}>0$ are such that \eqref{eq:low} holds and, in addition, also 
\[
\sup_{(t,x)\in\overline{\cI_n\cap U}}\big(|v^{\eps,\delta}_x(t,x)|- \alpha_0\big)\le-\tfrac{1}{2n},
\]
for all $\eps\in(0,\eps_{n,U})$ and $\delta\in(0,\delta_{n,U})$.

Restricting the PDE in \eqref{eq:os0} to the set $\cI_n$ we obtain 
\begin{align}\label{eq:os3}
(\partial_t+\cL-r)v^{\eps,\delta}=-h \quad\text{ in $\cC_n\cap\cI_n\cap U$,}
\end{align}
in the classical sense, for all $\eps\in(0,\eps_{n,U})$ and $\delta\in(0,\delta_{n,U})$. Differentiating with respect to $x$ yields
\begin{align*}
(\partial_t+\cG)v^{\eps,\delta}_x(t,x)+h_x(t,x)=0 \quad \text{for all $(t,x)\in\cC_n\cap\cI_n\cap U$,}
\end{align*}
also in the classical sense. Now, replacing $\tau$ in \eqref{eq:itovx} with 
$\sigma_n\coloneqq\inf\{s\ge 0:(t\!+\!s,Y^y_s)\notin \cI_n\}\wedge(T-t)$,
we obtain
\begin{align*}
v^{\eps,\delta}_x(t,y)= &\,\E\Big[\e^{-\int_0^{\bar{\rho}\wedge\sigma_n}\lambda(Y_s^y)\,\ud s}v^{\eps,\delta}_x(t+\bar\rho\wedge\sigma_n,Y^y_{\bar\rho\wedge\sigma_n})+\int_0^{\bar{\rho}\wedge\sigma_n}\!\e^{-\int_0^{s}\lambda(Y_u^y)\,\ud u}h_x(t+s,Y^y_{s})\,\ud s\Big],
\end{align*}
for all $\eps\in(0,\eps_{n,U})$ and $\delta\in(0,\delta_{n,U})$.

First we let $U\uparrow (0,T)\times\cO$, so that $\bar \rho_{n,U}\uparrow \rho_n$, and then $n\to\infty$. We use again monotone convergence and dominated convergence as in the first part of this proof to obtain 
\begin{align*}
v_x(t,y)= \E\Big[\e^{-\int_0^{\sigma_*}\lambda(Y_s^y)\,\ud s}\alpha_0\mathds{1}_{\{\sigma_*<\tau_a(t,y)\}}+\int_0^{\tau_a(t,y)\wedge\sigma_*}\!\e^{-\int_0^{s}\lambda(Y_u^y)\,\ud u}h_x(t+s,Y^y_{s})\,\ud s\Big],
\end{align*}
where we also used that $\sigma_n\uparrow\sigma_*=\inf\{s\ge 0: v_x(t+s,Y^y_s)\ge \alpha_0\}\wedge(T-t)$, thanks to continuity of $s\mapsto v_x(t+s,Y^y_s)$. 

Thus the equality holds in \eqref{eq:limvx3a} and $\sigma_*$ is optimal.
\end{proof}
\begin{remark}
The optimal stopping time $\sigma_*$ can be equivalently expressed as
\[
\sigma_*=\inf\{s\ge 0:Y_s^{y}\geq b(t+s)\}\wedge (T-t)
\]
with $b(t)$ defined as in \eqref{eq:uppbndry}.
\end{remark}
Thanks to the probabilistic representation of $v_x$ we are able to prove the next proposition. In turn, the latter will yield monotonicity of $t\mapsto b(t)$.
\begin{proposition}\label{prop:s->v_xdecr}
The mapping $t\mapsto v_{x}(t,y)$ is non-increasing for all $y\in\cO$.
\end{proposition}
\begin{proof}
Let $t,s\in[0,T]$ with $t>s$ and $y\in\cO$. Since $u\mapsto a(u)$ is non-decreasing and $(Y_u^{y})_{u\in[0,T]}$ is time-homogeneous, then $\tau_a(t,y)\leq \tau_a(s,y)$, $\P$-a.s. Thus
\begin{align}\label{eq:indic_ts}
\mathds{1}_{\{\tau<\tau_a(t,y)\}}\leq \mathds{1}_{\{\tau<\tau_a(s,y)\}},
\end{align}
for any stopping time $\tau\in\cT_0$.
Take arbitrary $\tau\in\cT_s$ and set $\tau_t=\tau\wedge(T-t)\in\cT_t$. Then we have
\begin{align*}
\begin{aligned}
v_x(t,y)&\le \E\Big[\e^{-\int_0^{\tau_t}\lambda(Y_u^y)\ud u}\alpha_{0}\mathds{1}_{\{\tau_t< \tau_a(t,y)\}}+\int_{0}^{\tau_t\wedge\tau_a(t,y)}\!\e^{-\int_0^u\lambda(Y_v^y)\ud v}h_x(t+u,Y_u^y)\,\ud u\Big].
\end{aligned}
\end{align*}
Now we make two observations. First we notice that 
\[
\int_{0}^{\tau_t\wedge\tau_a(t,y)}\!\e^{-\int_0^u\lambda(Y_v^y)\ud v}h_x(t+u,Y_u^y)\,\ud u\le \int_{0}^{\tau\wedge\tau_a(s,y)}\!\e^{-\int_0^u\lambda(Y_v^y)\ud v}h_x(s+u,Y_u^y)\,\ud u,
\]
because $0\le h_x(t\!+\!u,Y_u^y)\le h_x(s\!+\!u,Y_u^y)$ by Assumption \ref{ass:1}(ii) and $\tau_t\wedge\tau_a(t,y)\le \tau\wedge\tau_a(s,y)$. Second, we have equivalence of the events 
\[
\{\tau_t< \tau_a(t,y)\}=\{\tau< \tau_a(t,y)\},
\]
and on that event we also have $\tau_t=\tau$.
Therefore
\begin{align*}
& \e^{-\int_0^{\tau_t}\lambda(Y_u^y)\ud u}\alpha_{0}\mathds{1}_{\{\tau_t< \tau_a(t,y)\}}=\e^{-\int_0^{\tau}\lambda(Y_u^y)\ud u}\alpha_{0}\mathds{1}_{\{\tau< \tau_a(t,y)\}\cap\{\tau_t=\tau\}}\\
& \le\e^{-\int_0^{\tau}\lambda(Y_u^y)\ud u}\alpha_{0}\mathds{1}_{\{\tau< \tau_a(s,y)\}},
\end{align*}
where for the inequality we used \eqref{eq:indic_ts} and that $\alpha_0\ge 0$.

Combining the two observations we obtain
\begin{align*}
\begin{aligned}
v_x(t,y)\le
\E\Big[\e^{-\int_0^{\tau}\lambda(Y_u^y)\ud u}\alpha_{0}\mathds{1}_{\{\tau< \tau_a(s,y)\}}+\int_{0}^{\tau\wedge\tau_a(s,y)}\!\e^{-\int_0^u\lambda(Y_v^y)\ud v}h_x(s+u,Y_u^y)\,\ud u\Big].
\end{aligned}
\end{align*}
Since the inequality holds for any $\tau\in\cT_s$, then $v_x(t,y)\le v_x(s,y)$ as claimed.
\end{proof}

\begin{corollary}\label{cor:bproper}
Assume $\cM\neq\varnothing$. Then, the mapping $t\mapsto b(t)$ is non-decreasing and left-continuous on $(0,T)$ (as a map $(0,T)\to[\inf\cO,+\infty]$), with 
\begin{align*}
\lim_{t\to T}b(t)=+\infty.
\end{align*}
\end{corollary}

\begin{proof}
The mapping $t\mapsto b(t)$ is non-decreasing because, recalling $\cI=\{v_x<\alpha_0\}$, Proposition \ref{prop:s->v_xdecr} guarantees 
$(s,x)\in\cI\implies (t,x)\in\cI$ for all $0\le s<t\le T$. 

The set $\cM=\{v_x=\alpha_{0}\}$ is closed by continuity of $v_x$. Thus, for any increasing sequence $(t_n)_{n\in\N}\subset[0,T]$ converging to $t_0\in[0,T)$ we have that
\begin{align*}
\lim_{n\to\infty}(t_n,b(t_n))=(t_0,b(t_0-))\in\cM \Longrightarrow b(t_0-)\geq b(t_0),
\end{align*}
where $b(t_0-)=\lim_{n\to\infty}b(t_n)$. Since $b(t_n)\leq b(t_0)$ for all $n\in\N$ by monotonicity of $b$, we conclude that $b(t_0-)=b(t_0)$. Hence $b$ is left-continuous.

For the final claim, assume by contradiction that $\lim_{t\to T}b(t)\eqqcolon b(T-)<+\infty$. Since $b$ is non-decreasing and $v_x$ is continuous, we have $v_x(t,b(T-))=\alpha_0
$ for for all $t<T$. However, by continuity of $v_x$ on $[0,T]\times\cO$ we have $\lim_{(t,x)\to(T,x_0)}v_x(t,x)=0$ for any $x_0\in\cO$. Thus we reached a contradiction. 
\end{proof}
\begin{remark}\label{rem:b}
Without more specific assumptions on $\cO$ it is possible that $b(t)=+\infty$ for $t\in(t_0,T]$, for some $t_0\in[0,T)$. We provide conditions under which this is ruled out in Proposition \ref{prop:bfin}.
\end{remark}

It is also of interest to investigate conditions under which $b(t)>\inf\cO$. Proposition \ref{prop:a-rc} guarantees $0<a(t)<b(t)$ when $\cO=(0,\infty)$. It remains to consider the case $\cO=\R$ so that $\cC\cap\cI\neq\varnothing$ both when $\cO=\R$ and when $\cO=(0,\infty)$. The condition $\cC\cap\cI\neq\varnothing$ will appear in the construction of the optimal control in the next section. Here we provide some mild sufficient conditions for $b(t)>\inf\cO$ which, for example, are satisfied in the benchmark case of Remark \ref{rem:benchm}.
\begin{proposition}\label{prop:b>0}
For $t_0\in[0,T]$ assume that for every $s\in(0,T-t_0]$
\begin{align}\label{eq:b>0}
\limsup_{y\downarrow\inf\cO}\E\Big[\e^{-\int_0^s\lambda(Y^y_u)\ud u}\Big(h_x(t_0,Y^y_{s})-\alpha_0\lambda(Y^y_s)\Big)\Big]< 0. 
\end{align}
Then $b(t)>\inf\cO$ for all $t\in(t_0,T]$.
\end{proposition}
\begin{proof}
Arguing by contradiction, let us assume there is $t_1>t_0$ such that $b(t_1)=\inf\cO$. Then, by monotonicity it must be $b(t)=\inf\cO$ for $t\in[0,t_1]$ and, moreover, $a(t)=\inf\cO$ for $t\in[0,t_1]$. In particular, taking $\sigma=t_1-t_0$ in \eqref{eq:optstop} and using that $\tau_a(t_0,y)>t_1-t_0$, $\P$-a.s., yields 
\begin{align*}
\begin{aligned}
v_x(t_0,y)&\le \E\Big[\e^{-\int_0^{t_1-t_0}\lambda(Y^y_s)\ud s}\alpha_0+\int_0^{t_1-t_0}\e^{-\int_0^s\lambda(Y^y_u)\ud u}h_x(t_0+s,Y^y_s)\ud s\Big]\\
&=\alpha_0+\E\Big[\int_0^{t_1-t_0}\e^{-\int_0^s\lambda(Y^y_u)\ud u}\big(h_x(t_0+s,Y^y_s)-\alpha_0\lambda(Y^y_s)\big)\ud s\Big]\\
&\le\alpha_0+\int_0^{t_1-t_0}\E\Big[\e^{-\int_0^s\lambda(Y^y_u)\ud u}\Big(h_x(t_0,Y^y_{s})-\alpha_0\lambda(Y^y_s)\Big)\Big]\ud s,
\end{aligned}
\end{align*}
where we recall that $t\mapsto h_x(t,x)$ is non-increasing. Since $\lambda$ is bounded from below and $h_x$ has at most quadratic growth, reverse Fatou's lemma yields $\limsup_{y\to\inf\cO}v_x(t_0,y)<\alpha_0$ from \eqref{eq:b>0}. Then $b(t_0)>\inf\cO$ and, by monotonicity, $b(t_1)>\inf\cO$. Hence we reach a contradiction.
\end{proof}

Monotonicity and left-continuity of the boundary $b$ that separates $\cI$ from $\cM$ are sufficient to construct an optimal control in the game with value \eqref{eq:game}. Combining that with optimality of $\tau_*$ in \eqref{eq:tau*} we therefore obtain a saddle point in the game and a characterisation of the players' optimal strategies. This formally completes the main task of the paper. It is however interesting to also obtain further regularity properties of the value function and of the optimal control boundary. Such analysis requires slightly stronger assumptions and it is postponed to Section \ref{sec:further}. 

\section{A saddle point}\label{sec:saddle}

The candidate optimal control $\nu$ is one that keeps the controlled process $X^\nu$ inside the region $\cI=\{v_x<\alpha_0\}$ and it will be constructed as a solution to the so-called Skorohod reflection problem at the boundary $b$. First we construct such candidate control and then we prove its optimality. Let us start by introducing the notion of solution of reflecting SDE required in our setup.

\begin{definition}[{\bf Solution of reflecting SDE}]\label{def:RSP}
Given a non-decreasing, left-continuous function $f:[0,T]\to\cO$ and an initial point $(t,x)\in[0,T)\times\cO$, we say that a pair $(X^{\nu^f},\nu^f)$ characterises a solution of a reflecting SDE at $f(\,\cdot\,)$ with initial condition $(t,x)$ if:
\begin{itemize}
\item[(i)] The pair $(X^{\nu^f},\nu^f)$ is $\F$-adapted;
\item[(ii)] $X^{\nu^f\!;x}_{0}=f(t)\wedge x$, $s\mapsto X^{\nu^f\!;x}_s\in\cO$ is continuous and its dynamics is given by \eqref{eq:SDE};
\item[(iii)] $\nu^f_0=-(x-f(t))^+$ and $s\mapsto\nu^f_s$ is continuous and non-increasing for $s\in[0,T-t]$;
\item[(iv)] The following conditions hold $\P$-a.s.:
\begin{align}\label{eq:refcond}
\begin{aligned}
&X^{\nu^f\!;x}_s\le f(t+s),\quad\text{for $s\in[0,T-t]$},\\
&\int_{(0,T-t]}\mathds{1}_{\big\{X^{\nu^f\!;x}_s<f(t+s)\big\}}\ud \nu^f_s=0.
\end{aligned}
\end{align} 
\end{itemize} 
\end{definition}
The next lemma states the well-posedness of the reflecting SDE at a generic boundary $f$ (we use the abbreviation {\bf RSDE}-$f$). Its proof is provided in Appendix for completeness because other proofs\footnote{Burdzy et al.\ \cite{burdzy2009skorokhod} prove the Skorokhod reflection problem for Brownian motion along right/left-continuous moving boundaries. Here we solve the problem for a reflecting SDE.} that we are aware of would normally require continuity of the function $f$. We notice that our proof works under the assumption that the coefficients $\mu$ and $\sigma$ in \eqref{eq:SDE} be only locally Lipschitz, so we do not really need the additional assumptions made in Assumption \ref{ass:1a}.
\begin{lemma}\label{lem:SRP}
Instead of Assumption \ref{ass:1a}, let $\mu$ and $\sigma$ be locally Lipschitz with linear growth. Given a non-decreasing, left-continuous function $f:[0,T]\to\cO$ and an initial point $(t,x)\in[0,T)\times\cO$, there is a unique pair $(X^{\nu^f},\nu^f)$ that characterises the solution of {\bf RSDE}-$f$ with initial condition $(t,x)$, up to indistinguishability. Moreover, the pair $(X^{\nu^f\!;x},\nu^f)$ can be expressed as
\begin{align}\label{eq:RSDE}
\begin{aligned}
&X^{\nu^f\!;x}_s=x+\int_0^s \mu(X^{\nu^f\!;x}_u)\ud u+\int_0^s\sigma(X^{\nu^f\!;x}_u)\ud W_u+\nu^f_s,\\
&\nu^f_s=-\sup_{0\le u\le s}\Big(x+\int_0^u \mu(X^{\nu^f\!;x}_v)\ud v+\int_0^u\sigma(X^{\nu^f\!;x}_v)\ud W_v-f(t+u)\Big)^+,
\end{aligned}
\end{align}
for $s\in[0,T-t]$, $\P$-a.s., and it holds
\begin{align}\label{eq:nuL2}
\E\Big[\sup_{0\le s\le T-t}\big|X^{\nu^f\!;x}_s\big|^2\Big]+\E\big[\big|\nu^f_{T-t}\big|^2\big]<\infty.
\end{align}
\end{lemma}

Since the boundary $b$ of the set $\cI=\{v_x<\alpha_0\}$ is left-continuous and non-decreasing, we have a simple corollary of the lemma above.
\begin{corollary}\label{cor:soluSkoroPrb}
For any $(t,x)\in[0,T]\times\cO$ such that $b(t)>\inf\cO$, there exists a unique pair $(X^{\nu^*},\nu^*)$ that characterises the solution of {\bf RSDE}-$b$ in the sense of Definition \ref{def:RSP}. Moreover, such pair can be written as in \eqref{eq:RSDE} with $b$ instead of $f$.
\end{corollary}

Using the corollary we can now produce a saddle point for the game.
\begin{theorem}
Fix $(t,x)\in[0,T]\times\cO$ such that $b(t)>\inf\cO$, recall $\nu^*$ defined in Corollary \ref{cor:soluSkoroPrb} and $\tau_*$ defined in \eqref{eq:tau*}. The control $\nu^*$ is admissible, i.e., $\nu^*\in\cA_{t,x}$, and optimal in the sense that
\begin{align}\label{eq:optnu}
v(t,x)=\sup_{\tau\in\cT_t}\cJ_{t,x}(\tau,\nu^*).
\end{align}
Therefore, the pair $(\tau_*,\nu^*)\in\cT_t\times\cA_{t,x}$ yields a saddle point, i.e., 
\begin{align}
\cJ_{t,x}(\tau,\nu^*)\leq \cJ_{t,x}(\tau_*,\nu^*)\leq \cJ_{t,x}(\tau_*,\nu)
\end{align}
for all $(\tau,\nu)\in\cT_t\times\cA_{t,x}$.
\end{theorem}
\begin{remark}
Notice that on the right-hand side of \eqref{eq:optnu} the control $\nu^*$ is stopped at $\tau$. Therefore we could equivalently use the notation $\nu^*|_\tau\coloneqq(\nu^*_{t\wedge\tau})_{t\ge 0}$ to keep track more explicitly of the stopped paths.
\end{remark}

\begin{remark}
It is worth noticing that the structure of $\nu^*$ implies that the process $X^{\nu^*}$ does not jump upwards. Therefore, at equilibrium, the expression for $\tau_*$ simplifies to 
\[
\tau_*(\nu^*)=\inf\{s\ge 0:X^{\nu^*}_s\le a(t+s)\}\wedge(T-t).
\]
\end{remark}

\begin{proof}
Combining \eqref{eq:opttau} and \eqref{eq:optnu} it is easy to deduce $v(t,x)= \cJ_{t,x}(\tau^*(\nu^*),\nu^*)$ and the inequalities $\cJ_{t,x}(\tau,\nu^*) \le v(t,x)\le \cJ_{t,x}(\tau_*(\nu),\nu)$. Thus, if we prove \eqref{eq:optnu} the saddle point follows. 

The process $\nu^*$ is an admissible control because $X^{\nu^*}_s\in\cO$ for all $s\in[0,T-t]$ by Definition \ref{def:RSP}(ii) and Corollary \ref{cor:soluSkoroPrb}. 
Now recall the solution $v^{\eps,\delta}$ of the penalised problem \eqref{eq:PDE_pen} along with the bound \eqref{eq:boundpsi}. For simplicity of exposition and with no loss of generality here we take $\cO=\R$; when $\cO=(0,\infty)$ the argument is the same because $X^{\nu^*}$ does not reach zero in finite time. Fix $n\in\N$ and denote $\cK_n\coloneqq [0,T-\frac1n]\times[-n,n]$. Fix $(t,x)\in\cK_n$ and set $\rho^{*}_n=\inf\{s\ge 0:(t+s,X^{\nu^*\!;x}_s)\notin\cK_{2n}\}$. Assuming with no loss of generality that $b(t)> -2n$ we have $\rho^*_n>0$, $\P$-a.s. Since $s\mapsto\nu^*_s$ is continuous for $s\in(0,T-t]$, then $X^{\nu^*\!;x}_{\rho^*_n}=X^{\nu^*\!;x}_{\rho^*_n-}$, $\P$-a.s. We know that $v^{\eps,\delta}\in C^{1,2;\gamma}(\cK_{2n})$ and therefore an application of Dynkin's formula yields, for any $\tau\in\cT_t$
\begin{align*}
\begin{aligned}
v^{\eps,\delta}(t,x)&=\E\Big[\e^{-r(\tau\wedge\rho_n^*)}v^{\eps,\delta}(t\!+\!\tau\wedge\rho_n^*,X^{\nu^*\!;x}_{\tau\wedge\rho_n^*})\Big]\\
&\quad+\E\Big[\int_0^{\tau\wedge\rho_n^*}\e^{-rs}\big(h\!+\!\tfrac1\delta(g\!-\!v^{\eps,\delta})^+\!-\!\psi_\eps(|v^{\eps,\delta}_x|^2\!-\!\alpha^2_0)\big)(t\!+\!s,X^{\nu^*\!;x}_s)\ud s\Big]\\
&\quad-\E\Big[\int_{(0,\tau\wedge\rho_n^*]}\e^{-rs} v_x^{\eps,\delta}(t\!+\!s,X^{\nu^*\!;x}_s)\ud \nu^*_s+v^{\eps,\delta}(t,x\wedge b(t))-v^{\eps,\delta}(t,x)\Big],
\end{aligned}
\end{align*}
where we used that $\nu^*$ has at most a single jump at time zero. We can drop the term $\frac1\delta(g\!-\!v^{\eps,\delta})^+$ on the right-hand side and continue with an inequality. Since we are on a compact $v^{\eps,\delta}$ and $v^{\eps,\delta}_x$ are bounded. Moreover $\E[|\nu^*_{T-t}|^2]<\infty$ and \eqref{eq:boundpsi} holds. Then, letting $\eps,\delta\to 0$, we can pass the limits inside expectation by dominated convergence. 

Recalling the convergence of $v^{\eps,\delta}$ to $v$ in $C^{0,1;\gamma}(\cK_{2n})$ (see \eqref{eq:prp_pen_prb}) we obtain
\begin{align*}
\begin{aligned}
v(t,x)&\ge\E\Big[\e^{-r(\tau\wedge\rho_n^*)}v(t\!+\!\tau\wedge\rho_n^*,X^{\nu^*\!;x}_{\tau\wedge\rho_n^*})\Big]\\
&\quad+\E\Big[\int_0^{\tau\wedge\rho_n^*}\e^{-rs}\big(h\!-\!\limsup_{\eps,\delta\to 0}\psi_\eps(|v^{\eps,\delta}_x|^2\!-\!\alpha^2_0)\big)(t\!+\!s,X^{\nu^*\!;x}_s)\ud s\Big]\\
&\quad-\E\Big[\int_{(0,\tau\wedge\rho_n^*]}\e^{-rs} v_x(t\!+\!s,X^{\nu^*\!;x}_s)\ud \nu^*_s+v(t,x\wedge b(t))-v(t,x)\Big].
\end{aligned}
\end{align*}
From (iii) and (iv) in Definition \ref{def:RSP} we know that $v(t,x\wedge b(t))-v(t,x)=-\alpha_0(x-b(t))^+=\alpha_0\nu^*_0$ and 
\[
v_x(t\!+\!s,X^{\nu^*\!;x}_s)\ud \nu^*_s=1_{\{X^{\nu^*\!;x}_s=b(t+s)\}}v_x(t\!+\!s,X^{\nu^*\!;x}_s)\ud \nu^*_s=\alpha_0\ud \nu^*_s\quad\text{for }s\in(0,T-t].
\]
Since the obstacle condition $v\ge g$ also holds, then we have
\begin{align}\label{eq:optnu0}
\begin{aligned}
v(t,x)&\ge\E\Big[\e^{-r(\tau\wedge\rho_n^*)}g(t\!+\!\tau\wedge\rho_n^*)\Big]\\
&\quad+\E\Big[\int_0^{\tau\wedge\rho_n^*}\e^{-rs}\big(h\!-\!\limsup_{\eps,\delta\to 0}\psi_\eps(|v^{\eps,\delta}_x|^2\!-\!\alpha^2_0)\big)(t\!+\!s,X^{\nu^*\!;x}_s)\ud s-\int_{[0,\tau\wedge\rho_n^*]}\e^{-rs}\alpha_0\ud \nu^*_s\Big].
\end{aligned}
\end{align}

For the term containing the limsup we need the following fact, which we prove at the end:
\begin{align}\label{eq:condFS1}
\P\Big((t+s,X^{\nu^*\!;x}_s)\in \cI,\,\text{ for a.e. } s\in[0,T-t]\Big)=1,
\end{align}
where we recall that $\cI=\{v_x<\alpha_{0}\}$. Fix $\omega\in\Omega$ outside of a null set. Condition \eqref{eq:condFS1} tells us that for a.e.\ $s\in[0,\tau(\omega)\wedge\rho^*_n(\omega)]$ there is $\eta=\eta_{\omega,s}\in(0,\alpha_0/2)$ such that $v_x(t+s,X^{\nu^*\!;x}_s(\omega))\le \alpha_0-\eta$. Then, the uniform convergence of $v_x^{\eps,\delta}$ to $v_x$ implies that for sufficiently small $\eps$ and $\delta$ we have $v^{\eps,\delta}_x(t+s,X^{\nu^*\!;x}_s(\omega))\le \alpha_0-\eta/2$. Hence
\[
\limsup_{\eps,\delta\to 0}\psi_\eps(|v^{\eps,\delta}_x|^2\!-\!\alpha^2_0)\big)(t\!+\!s,X^{\nu^*\!;x}_s(\omega))=0.
\]
Since this procedure holds almost surely for a.e.\ $s\in[0,\tau\wedge\rho^*_n]$, we conclude that \eqref{eq:optnu0} reads
\begin{align*}
\begin{aligned}
v(t,x)&\ge\E\Big[\e^{-r(\tau\wedge\rho_n^*)}g(t\!+\!\tau\wedge\rho_n^*)\!+\!\int_0^{\tau\wedge\rho_n^*}\e^{-rs}h(t\!+\!s,X^{\nu^*\!;x}_s)\ud s-\int_{[0,\tau\wedge\rho_n^*]}\e^{-rs}\alpha_0\ud \nu^*_s\Big].
\end{aligned}
\end{align*}
Letting $n\to\infty$ we have $\rho^*_n\uparrow T-t$. Recalling that $\ud\nu^*_s=-\ud |\nu^*|_s$ because $\nu^*$ is decreasing, we get
\begin{align*}
\begin{aligned}
v(t,x)&\ge\E\Big[\e^{-r\tau}g(t\!+\!\tau)\!+\!\int_0^{\tau}\e^{-rs}h(t\!+\!s,X^{\nu^*\!;x}_s)\ud s+\int_{[0,\tau]}\e^{-rs}\alpha_0\ud |\nu^*|_s\Big],
\end{aligned}
\end{align*}
by Fatou's lemma. Since $\tau\in\cT_t$ was arbitrary and $v(t,x)\le \sup_{\tau}\cJ_{t,x}(\tau,\nu^*)$ we obtain \eqref{eq:optnu}. Hence $(\tau_*,\nu^*)$ is a saddle point.

It remains to prove \eqref{eq:condFS1}. 
For every $s\in[0,T-t]$ we have $(t+s,X_s^{\nu^*\!;x})\in\overline\cI$, $\P$-a.s.,\ because $(X_s^{\nu^*\!;x},\nu^*_s)$ is solution of {\bf RSDE}-$b$ (see Corollary \ref{cor:soluSkoroPrb}). Next we prove that $\{s\in[0,T-t]:X^{\nu^*\!;x}_s(\omega)=b(t+s)\}$ has got zero Lebesgue measure for a.e.\ $\omega\in\Omega$. 
Denote $b_s^t\coloneqq b(t+s)$ and set $Z_s\coloneqq X_s^{\nu^*\!;x}-b^t_s$ for $s\in(0,T-t]$. Taking $Z_0=x\wedge b(t)-b(t)$, the process $Z$ is a c\`agl\`ad semi-martingale on $[0,T-t]$ with the following Ito's representation:
\begin{align}\label{eq:Z}
Z_s=x\wedge b(t)+\int_0^s \mu(X^{\nu^*\!;x}_u)\ud u+\int_0^s\sigma(X^{\nu^*\!;x}_u)\ud W_u+\nu_s^*-b_s^t,\quad s\in[0,T-t]. 
\end{align}
Notice that since $(b^t_s)_{s\in[0,T-t]}$ is monotone $\langle Z\rangle^c_s=\int_0^s\sigma^2(X^{\nu^*\!;x}_u)\ud u$ is the continuous part of the quadratic variation (cf.\ \cite[Sec.\ VII.44]{dellacherie1982probabilities}). 

Set $\bar Z_s\coloneqq Z_{s+}$ so that $\bar Z$ is a c\`adl\`ag semimartingale, with continuous part of the quadratic variation $\langle \bar Z\rangle^c_u=\langle Z\rangle^c_u$, $u\in[0,T-t]$. For $y\in\cO$ we denote by $(\bar L^y_s)_{s\in[0,T-t]}$ the semi-martingale local time\footnote{Here we follow the convention in \cite[Ch.\ IV]{protter2005stochastic}, where $\sign(x)=-1$ for $x\le 0$ and $\sign(x)=1$ for $x> 0$. Notice that, among other things, this choice leads to $2L^0(W)=L^0(|W|)$ for the local time at zero of Brownian motion and its absolute value.} at $y$ of the process $\bar Z$. Then, \cite[Cor.\ 1 of Thm.\ IV.70]{protter2005stochastic} gives us the occupation time formula: for any bounded Borel-measurable function $g:\cO\to \R$,
\begin{align*}
\int_0^s\!g(Z_u)\,\ud \langle Z\rangle^c_u=\int_0^s\!g(\bar Z_{u-})\,\ud \langle \bar Z\rangle^c_u=\int_0^s\!g(\bar Z_{u})\,\ud \langle \bar Z\rangle^c_u=\int_{-\infty}^{\infty}\! g(y)\bar L^y_s\ud y,\quad\text{$\P$-a.s.},
\end{align*}
where the second equality holds by continuity of $t\mapsto\langle\bar Z\rangle^c_t$. Using this formula we obtain
\begin{align}\label{eq:occtime}
\begin{aligned}
\int_0^s\!\mathds{1}_{\{Z_u=0\}}\,\ud \langle Z\rangle^c_u&\,=\lim_{n\to\infty}\int_0^s\!\mathds{1}_{\{Z_u\in[-\frac{1}{n},\frac{1}{n}]\}}\,\ud \langle Z\rangle^c_u
=\lim_{n\to\infty}\int_{-\frac{1}{n}}^{\frac{1}{n}}\! \bar L^y_s\ud y=0,\quad\text{$\P$-a.s.},
\end{aligned}
\end{align}
where for the final equality we use that $\bar L^y_s<\infty$, $\P$-a.s., for every $y\in\R$. Since $\langle Z\rangle^c_s=\int_0^s\sigma^2(X^{\nu^*\!;x}_u)\ud u$ and $\sigma(x)>0$ for $x\in\cO$, we conclude from \eqref{eq:occtime} that 
\[
\cL eb\big(\{s\in[0,T-t]:Z_s=0\}\big)=\cL eb\big(\{s\in[0,T-t]:X^{\nu^*\!;x}_s=b(t+s)\}\big)=0, \quad\text{$\P$-a.s.},
\]
as claimed.
\end{proof}

\section{Continuity of the optimal boundaries}\label{sec:further}

In this section we show continuity of the optimal boundaries $t\mapsto a(t)$ and $t\mapsto b(t)$. We recall that $t\mapsto a(t)$ is non-decreasing and right-continuous with $a(t)\le \underline{\Theta}(t)$ for all $t\in[0,T)$ (cf.\ Proposition \ref{prop:a-rc}). Instead, $t\mapsto b(t)$ is non-decreasing and left-continuous, but possibly $b(t)=+\infty$ for $t\in(t_0,T]$ and some $t_0$ (Corollary \ref{cor:bproper} and Remark \ref{rem:b}). We also recall that $\cS\cap\cM=\varnothing$, so that $a(t)<b(t)$ for all $t\in[0,T)$ such that $b(t)>\inf\cO$. Let us start by showing continuity of $a$. 

\begin{proposition}\label{prop:cont-a}
Given $t_0\in(0,T)$, if $h_x(t_0,a(t_0))>0$ then $t\mapsto a(t)$ is continuous at $t_0$.
\end{proposition}
\begin{proof}
Since we know that $a$ is right-continuous, we now need to prove that it is left-continuous. We follow ideas as in \cite[Sec.\ 3]{de2015note}. By contradiction assume there exists $t_0\in(0,T)$ such that $h_x(t_0,a(t_0))>0$ but $a(t_0-)<a(t_0)$. Recall that $\cS\cap\cM=\varnothing$. Then, from \eqref{eq:PDE} and monotonicity of $t\mapsto a(t)$ we know that for every $x\in\cQ_0\coloneqq(a(t_0-),a(t_0))$ and any $s\in[t_0-\eps,t_0)$, with $\eps>0$ sufficiently small, we have
\begin{align}\label{eq:cont}
\begin{aligned}
0&=v_t(s,x)+(\cL v)(s,x)-rv(s,x)+h(s,x)\\
&\le \dot g(s)+(\cL v)(s,x)-rg(s)+h(s,x)=(\cL v)(s,x)+\Theta(s,x).
\end{aligned}
\end{align}
where the inequality holds because $v(s,x)\ge g(s)$ and $v_t(s,x)\le \dot g(s)$ (cf.\ Step 1 in the proof of Proposition \ref{prop:a-rc}).

Let $\varphi\in C^\infty_c(\cQ_0)$ be positive with $\int\varphi(x)\ud x=1$. Multiplying \eqref{eq:cont} by $\varphi$ and integrating we get
\begin{align*}
0&\le \int_{a(t_0-)}^{a(t_0)}\varphi(x)(\cL v)(s,x)\ud x+\int_{a(t_0-)}^{a(t_0)}\varphi(x)\Theta(s,x)\ud x\\
&=\int_{a(t_0-)}^{a(t_0)}(\cL^*\varphi)(x)v(s,x)\ud x+\int_{a(t_0-)}^{a(t_0)}\varphi(x)\Theta(s,x)\ud x,
\end{align*}
where $\cL^*$ is the adjoint operator of $\cL$.
Taking limits as $s\uparrow t_0$ yields
\begin{align*}
0\le &\int_{a(t_0-)}^{a(t_0)}(\cL^*\varphi)(x)v(t_0,x)\ud x+\int_{a(t_0-)}^{a(t_0)}\varphi(x)\Theta(t_0,x)\ud x\\
=&\int_{a(t_0-)}^{a(t_0)}(\cL^*\varphi)(x)g(t_0)\ud x+\int_{a(t_0-)}^{a(t_0)}\varphi(x)\Theta(t_0,x)\ud x
=\int_{a(t_0-)}^{a(t_0)}\varphi(x)\Theta(t_0,x)\ud x\le 0,
\end{align*}
where the final inequality holds because $\Theta(t_0,x)\le 0$ for $x\le a(t_0)$ (cf.\ Step 3 in the proof of Proposition \ref{prop:a-rc}). By arbitrariness of $\varphi$ it must then be $\Theta(t_0,x)=0$ for all $x\in(a(t_0-),a(t_0))$. This contradicts the assumption that $\Theta_x(t_0,a(t_0))=h_x(t_0,a(t_0))>0$, because by continuity of $h_x$ there must be $\delta>0$ for which $h_x(t,x)>0$ for $x\in(a(t_0)-\delta,a(t_0)]$. Hence, it must be $a(t_0-)=a(t_0)$.
\end{proof}
The next corollaries follow immediately.
\begin{corollary}
If $h_x(T,\underline{\Theta}(T))>0$ then $a(T-)=\underline{\Theta}(T)$.
\end{corollary}
\begin{proof}
Recall that $a(t)\le \underline\Theta(t)$ for all $t\in[0,T)$. Let us assume by contradiction that $a(T-)<\underline{\Theta}(T)$. Then \eqref{eq:cont} holds for every $x\in(a(T-),\underline\Theta(T))$ and all $s\in[T-\eps,T)$ for a sufficiently small $\eps>0$. Therefore, repeating the rest of the proof of Proposition \ref{prop:cont-a} with $\underline\Theta(T)$ instead of $a(t_0)$ and $a(T-)$ instead of $a(t_0-)$ leads to 
\[
\int_{a(T-)}^{\underline\Theta(T)}\varphi(x)\Theta(T,x)\ud x= 0.
\]
That is impossible because $\Theta_x(T,x)=h_x(T,x)>0$ for $x\in (\underline\Theta(T)-\delta,\underline\Theta(T)]$ and suitable $\delta>0$.
\end{proof}

\begin{corollary}
If $h_x(t,x)>0$ for $x\le \underline\Theta(t)$ and all $t\in[0,T]$, then $t\mapsto a(t)$ is continuous on $[0,T]$ with $a(T-)=\underline{\Theta}(T)$.
\end{corollary}

Next we show conditions under which $b(t)<\infty$ for $t\in[0,T)$. The next proposition requires a slightly abstract sufficient condition, but after the proof we will provide easy sufficient conditions that imply the more abstract one.
\begin{proposition}\label{prop:bfin}
Assume that for all $t\in[0,T)$
\begin{align}\label{eq:suffbO}
\liminf_{y\to\infty}\E\Big[\int_0^{\tau_a(t,y)}\e^{-\int_0^s\lambda(Y^y_u)\ud u}h_x(t+s,Y^y_s)\ud s\Big]> \alpha_0.
\end{align}
Then $b(t)<\infty$ for all $t\in[0,T)$.
\end{proposition}
\begin{proof}
Let us assume by contradiction that there exists $t_0\in[0,T)$ such that $b(t_0)=\infty$. Then it must be $v_x(t,y)<\alpha_0$ for all $(t,y)\in[t_0,T]\times\cO$ by monotonicity of $v_x(\,\cdot\,,y)$. In particular, for any $(t,y)\in[t_0,T)\times\cO$ we have $\sigma_*(t,y)=T-t$, $\P$-a.s.\ by \eqref{eq:tau*vx}. Then \eqref{eq:optstop} yields
\begin{align*}
\alpha_0>v_x(t,y)=\E\Big[\int_{0}^{\tau_a(t,y)}\!\e^{-\int_0^s\lambda(Y_u^y)\ud u}h_x(t+s,Y_s^y)\,\ud s\Big].
\end{align*}
Letting $y\uparrow \infty$ and using \eqref{eq:suffbO} we reach a contradiction.
\end{proof}
\begin{remark}
An easy sufficient condition for \eqref{eq:suffbO} is that
$\lim_{y\to\infty}h_{x}(s,y)=\infty$ for all $s\in[0,T]$, because by Fatou's lemma 
\begin{align*}
\liminf_{y\to\infty}\E\Big[\int_0^{\tau_a(t,y)}\!\!\!\e^{-\int_0^s\lambda(Y^y_u)\ud u}h_x(t\!+\!s,Y^y_s)\ud s\Big]
\ge \E\Big[\liminf_{y\to\infty}\int_0^{\tau_a(t,y)}\!\!\!\e^{-\int_0^s\lambda(Y^y_u)\ud u}h_x(t\!+\!s,Y^y_s)\ud s\Big]=\infty,
\end{align*}
upon noticing that $\tau_a(t,y)\uparrow T-t$ and $\lambda(y)$ is bounded from above (cf.\ Assumption \ref{ass:1a}(i)). This is satisfied in our benchmark example from Remark \ref{rem:benchm}.
\end{remark}

For the continuity of $b$, which we prove in the next proposition, we require some additional regularity of $\mu$ and $h$. With a slight abuse of notation it is convenient to interpret $\mu$ as a function from $[0,T)\times\cO$ to $\R$.
\begin{proposition}
Let $t_0\in(0,T)$ and let $U_0$ be an open neighbourhood of $(t_0,b(t_0))$. If for some $\gamma\in(0,1)$ we have $\mu\in C^{2;\gamma}(U_0)$, $h\in C^{0,2;\gamma}(U_0)$, and either $h_{xx}(t_0,b(t_0))>0$ or $\mu_{xx}(b(t_0))>0$ (or both), then $t\mapsto b(t)$ is continuous at $t_0$.
\end{proposition}
\begin{proof}
We know from Corollary \ref{cor:bproper} that $b$ is non-decreasing and left-continuous. Therefore, it only remains to prove that it is also right-continuous. Arguing by contradiction we assume that $\lim_{t\downarrow t_0}b(t)=b(t_0+)>b(t_0)$ for some $t_0\in(0,T)$. 
There exist $\eps>0$ and two points $b_1,b_2$ such that $b(t_0)<b_1<b_2<b(t_0+)$ and $[b_1,b_2]\times [t_0,t_0+\varepsilon] \subset U_0$. It follows from Theorem \ref{thm:v_xVI} that $v_x$ is a classical solution of
\begin{align*}
\big(\partial_t v_x +\cG v_x\big)(t,x)=-h_x(t,x),\quad(t,x)\in U_0.
\end{align*}
Thanks to interior regularity of solutions to parabolic PDEs we can differentiate the equation above with respect to $x$. Noticing that $\sigma_{xx}\equiv 0$, we obtain
\begin{align}\label{eq:PDEvxxproof}
\big(\partial_tv_{xx} +\cH v_{xx}\big)(t,x)=-h_{xx}(t,x)-\mu_{xx}(x)v_x(t,x),\quad(t,x)\in U_0,
\end{align}
where $\cH$ is defined on functions $\varphi\in C^2(\R)$ by 
\begin{align*}
(\cH\varphi)(x)\coloneqq \frac{\sigma^2(x)}{2}\varphi_{xx}(x)+\Big[2\sigma(x)\sigma_x(x)+\mu(x)\Big]\varphi_x(x)-\Big[r-2\mu_x(x)-\sigma_x^2(x)\Big]\varphi(x).
\end{align*}

Let $\varphi\in C^{\infty}_{c}((b_1,b_2))$ be non-negative and $\int\varphi(x)\ud x=1$. Multiplying \eqref{eq:PDEvxxproof} by $\varphi$ and integrating over $(t_0,t_0+\eps)\times(b_1,b_2)$ we obtain
\begin{align*}
\begin{aligned}
&\int_{t_0}^{t_0+\varepsilon}\!\!\int_{b_1}^{b_2}\!\Big[\big(\partial_tv_{xx}\! +\!\cH v_{xx}\big)(t,x)\Big]\varphi(x)\,\ud x\,\ud t=\!-\!\int_{t_0}^{t_0+\varepsilon}\!\!\int_{b_1}^{b_2}\! \big[h_{xx}(t,x)\!+\!\mu_{xx}(x)v_x(t,x)\big]\varphi(x)\,\ud x\,\ud t.
\end{aligned}
\end{align*}
If $h_{xx}(t_0,b(t_0))>0$, then by continuity on $U_0$ we can assume with no loss of generality that $h_{xx}(t,x)\ge \beta_0>0$ for $(t,x)\in[t_0,t_0+\eps]\times[b_1,b_2]$. If instead $\mu_{xx}(b(t_0))>0$, recalling that $v_x(t_0,b(t_0))=\alpha_0$ we can assume with no loss of generality (by continuity on $U_0$) that $\mu_{xx}(x)v_x(t,x)\ge \beta_0>0$ for $(t,x)\in[t_0,t_0+\eps]\times[b_1,b_2]$. 
Then, recalling that also $\varphi\ge 0$ integrates to one, we obtain
\begin{align}\label{eq:intbypartH}
\begin{aligned}
&\int_{t_0}^{t_0+\varepsilon}\!\!\int_{b_1}^{b_2}\!\Big[\big(\partial_tv_{xx}\! +\!\cH v_{xx}\big)(t,x)\Big]\varphi(x)\,\ud x\,\ud t\le \!-\beta_0\eps.
\end{aligned}
\end{align}
For the first term in the integral on the left-hand side of \eqref{eq:intbypartH}, we have
\begin{align}\label{eq:BPjump2}
\begin{aligned}
&\int_{t_0}^{t_0+\varepsilon}\!\int_{b_1}^{b_2}\!\partial_tv_{xx}(t,x)\varphi(x)\,\ud x\,\ud t=-\int_{t_0}^{t_0+\varepsilon}\!\int_{b_1}^{b_2}\!\partial_tv_x(t,x)\varphi_{x}(x)\,\ud x\,\ud t\\
&=-\int_{b_1}^{b_2}\!(v_x(t_0+\varepsilon,x)-\alpha_0)\varphi_{x}(x)\,\ud x=\int_{b_1}^{b_2}\!v_{xx}(t_0+\varepsilon,x)\varphi(x)\,\ud x\geq 0,
\end{aligned}
\end{align}
where, first we integrated by parts, then used Fubini's theorem and $v_x(t_0,x)=\alpha_0$ for $x\in (b_1,b_2)$ and, finally, integrated by parts again. The final inequality follows by convexity of $v(t_0+\eps,\cdot)$ (cf.\ Proposition \ref{prop:vconvex}).

Plugging \eqref{eq:BPjump2} into the left-hand side of \eqref{eq:intbypartH}, we have
\begin{align*}
\begin{aligned}
-\beta_0\varepsilon\ge \int_{t_0}^{t_0+\varepsilon}\!\!\int_{b_1}^{b_2}\!\big(\cH v_{xx}\big)(t,x)\varphi(x)\,\ud x\,\ud t=-\int_{t_0}^{t_0+\varepsilon}\!\int_{b_1}^{b_2}\!v_{x}(t,x)\frac{\partial}{\partial x}\Big[\big(\cH^*\varphi\big)(x)\Big]\,\ud x\,\ud t,
\end{aligned}
\end{align*}
where $\cH^*$ is the adjoint of $\cH$. Dividing both sides by $\varepsilon$ and letting $\varepsilon\downarrow0$, we have
\begin{align*}
\beta_0\le \int_{b_1}^{b_2}\!v_{x}(t_0,x)\frac{\partial}{\partial x}\Big[\big(\cH^*\varphi\big)(x)\Big]\,\ud x=\alpha_0\int_{b_1}^{b_2}\!\frac{\partial}{\partial x}\Big[\big(\cH^*\varphi\big)(x)\Big]\,\ud x= 0,
\end{align*}
where the final equality holds because $(\cH^*\varphi)(b_1)=(\cH^*\varphi)(b_2)=0$.
Hence, we reach a contradiction and the function $b$ must be right-continuous. 
\end{proof}
A simple corollary follows.
\begin{corollary}
If $\mu\in C^{2;\gamma}_{\ell oc}(\cO)$, $h\in C^{0,2;\gamma}_{\ell oc}([0,T)\times\cO)$ and, for each $t\in[0,T)$, either $h_{xx}(t,b(t))>0$ or $\mu_{xx}(b(t))>0$ (or both) then the mapping $t\to b(t)$ is continuous on $[0,T)$.
\end{corollary}

\appendix

\section{}
\subsection{Proof of Lemma \ref{lem:convexity}}\label{app:comparison}
We notice that for this proof we do not need $\mu_x$ to be bounded from above as indicated in Assumption \ref{ass:1a}. It suffices that $\mu$ be Lipschitz on $[-n,n]$ when $\cO=\R$ and on $[0,n]$ when $\cO=(0,\infty)$ for all $n\in\N$. We first prove the lemma assuming that $\mu$ is Lipschitz continuous and we later extend it to locally Lipschitz $\mu$. Let $\mu:\cO\to\R$ be Lipschitz and convex. Let $x_1,x_2\in\cO$ and $t=0$. Let also $\lambda\in(0,1)$ and denote $x_\lambda=\lambda x_1+(1-\lambda)x_2$. Let $\nu^{1}\in\cA_{0,x_1}$ and $\nu^{2}\in\cA_{0,x_2}$ and set $\nu^{\lambda}=\lambda\nu^1+(1-\lambda)\nu^2$. Finally, denote $X^1=X^{\nu_1;x_1}$, $X^2=X^{\nu_2;x_2}$ and $X^\lambda=X^{\nu^\lambda;x_\lambda}$ for simplicity.

For $Y^{\lambda}\coloneqq\lambda X^{1}+(1-\lambda)X^{2}$, the process $X^\lambda_{\cdot\wedge\theta_\lambda}-Y^\lambda_{\cdot\wedge\theta_\lambda}$ is a {\em continuous} semi-martingale with dynamics 
\begin{align*}
X^{\lambda}_{t\wedge\theta_\lambda}\!\!-\!Y^{\lambda}_{t\wedge\theta_\lambda}\!\!=\!\!\!\int_{0}^{t\wedge\theta_\lambda}\!\!\!\!\big(\mu(X_s^\lambda)\!-\!\lambda\mu(X_s^{1})\!-\!(1\!-\!\lambda)\mu(X_s^{2})\big)\ud s\!+\!\int_{0}^{t\wedge\theta_\lambda}\!\!\!\!\big(\sigma(X_s^\lambda)\!-\!\lambda\sigma(X_s^{1})\!-\!(1\!-\!\lambda)\sigma(X_s^{2})\big)\ud W_s.
\end{align*}
Then, by Tanaka's formula \cite[Thm.\ IV.68]{protter2005stochastic}, 
\begin{align}\label{eq:convexsde}
(X^{\lambda}_{t\wedge\theta_\lambda}-Y^{\lambda}_{t\wedge\theta_\lambda})^+&=\int_{0}^{t\wedge\theta_\lambda}\!\!\mathds{1}_{\{X_s^\lambda>Y_s^\lambda\}}\big(\mu(X_s^\lambda)\!-\!\lambda\mu(X_s^{1})\!-\!(1\!-\!\lambda)\mu(X_s^{2})\big)\,\ud s\notag\\
&\,+\int_{0}^{t\wedge\theta_\lambda}\!\!\mathds{1}_{\{X_s^\lambda>Y_s^\lambda\}}\big(\sigma(X_s^\lambda)\!-\!\lambda\sigma(X_s^{1})\!-\!(1\!-\!\lambda)\sigma(X_s^{2})\big)\,\ud W_s\!+\!\tfrac{1}{2}L_{t\wedge\theta_\lambda}^0(X^\lambda\!-\!Y^\lambda),
\end{align}
where $L^0(X^\lambda-Y^\lambda)$ is the semi-martingale local-time at zero of $X^\lambda-Y^\lambda$. 

From \cite[Cor. 3 of Thm.\ 75]{protter2005stochastic} we know that 
\[
L_{t\wedge\theta_\lambda}^0(X^\lambda-Y^\lambda)=\lim_{\varepsilon\downarrow0}\frac{1}{\varepsilon}\int_0^{t\wedge\theta_\lambda}\!\mathds{1}_{\{0\leq X^\lambda_s-Y^\lambda_s\leq\varepsilon\}}\big(\sigma(X_s^\lambda)\!-\!\lambda\sigma(X_s^{1})\!-\!(1\!-\!\lambda)\sigma(X_s^{2})\big)^2\,\ud s.
\]
Using $\sigma(x)=\sigma_0+\sigma_1 x$ with $\sigma_1+\sigma_0>0$ and $\sigma_0\cdot\sigma_1=0$ (cf.\ Assumption \ref{ass:1a}(iii)) we obtain
\begin{align}\label{eq:loctimebnd}
\begin{aligned}
L_{t\wedge\theta_\lambda}^0(X^\lambda-Y^\lambda)=&\,\lim_{\varepsilon\downarrow0}\frac{1}{\varepsilon}\int_0^{t\wedge\theta_\lambda}\!\mathds{1}_{\{0\le X^\lambda_s-Y^\lambda_s\le\varepsilon\}}\big(\sigma(X_s^\lambda)\!-\!\lambda\sigma(X_s^{1})\!-\!(1\!-\!\lambda)\sigma(X_s^{2})\big)^2\,\ud s\\
\le&\,\lim_{\varepsilon\downarrow0}\frac{\sigma_1^2}{\varepsilon}\int_0^{t\wedge\theta_\lambda}\!\mathds{1}_{\{0\le X^\lambda_s-Y^\lambda_s\le\varepsilon\}}\big(X_s^\lambda-Y^\lambda_s\big)^2\,\ud s
\leq\lim_{\varepsilon\downarrow0}(t\sigma_1^2\varepsilon)=0,\qquad\text{$\P$-a.s.}
\end{aligned}
\end{align}

We take expectations on both sides of \eqref{eq:convexsde} and use \eqref{eq:loctimebnd}. The stochastic integral is a (local) martingale and it disappears (up to standard localisation, if needed). By convexity of $\mu$ we have $\mu(Y_s^{\lambda})\le\lambda\mu(X_s^{1})+(1-\lambda)\mu(X_s^{2})$ for all $s\in[\![0,t\wedge\theta_\lambda)\!)$. Then, 
\begin{align}
\E\big[(X^{\lambda}_{t\wedge\theta_\lambda}-Y^{\lambda}_{t\wedge\theta_\lambda})^+\big]=&\,\E\Big[\int_{0}^{t\wedge\theta_\lambda}\mathds{1}_{\{X_s^\lambda>Y_s^\lambda\}}\big(\mu(X_s^\lambda)-\lambda\mu(X_s^{1})-(1-\lambda)\mu(X_s^{2})\big)\,\ud s\Big]\\
\leq&\,\E\Big[\int_{0}^{t\wedge\theta_\lambda}\mathds{1}_{\{X_s^\lambda>Y_s^\lambda\}}\big(\mu(X_s^\lambda)-\mu(Y_s^{\lambda})\big)\,\ud s\Big]\\
\leq&\,\E\Big[L\int_{0}^{t\wedge\theta_\lambda}\mathds{1}_{\{X_s^\lambda>Y_s^\lambda\}}\big|X_s^\lambda-Y_s^{\lambda}\big|\,\ud s\Big]\\
=&\E\Big[L\int_{0}^{t\wedge\theta_\lambda}\big(X_s^\lambda-Y_s^{\lambda}\big)^+\,\ud s\Big]\le L\int_{0}^{t}\E\Big[\big(X_{s\wedge\theta_\lambda}^\lambda-Y_{s\wedge\theta_\lambda}^{\lambda}\big)^+\Big]\,\ud s ,
\end{align}
where the second inequality is by the Lipschitz continuity of $\mu$ with constant $L$. By an application of Gronwall's lemma to the function $f(t)\coloneqq\E[(X^\lambda_{t\wedge\theta_\lambda}-Y^\lambda_{t\wedge\theta_\lambda})^+]$, we obtain $f(t)=0$
for all $t\in[0,T]$. Since the process $X^{\lambda}-Y^{\lambda}$ is continuous, then $X^{\lambda}_{t\wedge\theta_\lambda}\leq Y^{\lambda}_{t\wedge\theta_\lambda}$, for all $t\in[0,T]$, $\P$-a.s.

Now we relax the Lipschitz assumption on $\mu$. For simplicity and with no loss of generality we consider $\cO=\R$ so that $\theta_\lambda=T$. Assume that $\mu$ is only locally Lipschitz with linear growth. Then 
\begin{align}\label{eq:Xzsquare}
\E\Big[\sup_{0\leq t\leq T}|X_t^{\nu;y}|^2\Big]\le c(1+|y|^2+\E[\nu^2_T])<\infty,\quad\text{for any $y\in\cO$ and $\nu\in\cA_{0,y}$,}
\end{align}
by standard estimates \cite[Cor.\ 2.5.10]{krylov1980controlled}. Letting 
$\tau_n^{\nu;y}\coloneqq \inf\{s\ge 0: X^{\nu;y}_s\notin (-n,n)\}\wedge T$,
from \eqref{eq:Xzsquare} we have $\P(\tau^{\nu;y}_n\leq T)\to 0$ as $n\to\infty$. 
Fix $n\in\N$ and denote $\tau_n=\tau_n^{\nu^1;x_1}\wedge\tau_n^{\nu^2;x_2}\wedge\tau_n^{\nu^\lambda;x_\lambda}$. Repeating the first part of the proof with $(X_t^{\lambda})_{t \in[0,T]}$ and $(Y_t^{\lambda})_{t \in[0,T]}$ replaced by $(X_{t\wedge\tau_n}^{\lambda})_{t \in[0,T]}$ and $(Y_{t\wedge\tau_n}^{\lambda})_{t \in[0,T]}$, respectively, we have
$\E\big[(X^{\lambda}_{t\wedge\tau_n}-Y^{\lambda}_{t\wedge \tau_n})^+\big]=0$,
for all $t\in[0,T]$. Thus, by Fatou's lemma and up to selecting a subsequence $(\tau_{n_k})_{k\in\N}$ such that $\tau_{n_k}\uparrow T$, $\P$-a.s., we have
\[
\E\big[(X^{\lambda}_{t}-Y^{\lambda}_{t})^+\big]\leq \lim_{k\to\infty}\E\big[(X^{\lambda}_{t\wedge\tau_{n_k}}-Y^{\lambda}_{t\wedge \tau_{n_k}})^+\big]=0.
\]
By continuity of the process $X^{\lambda}-Y^{\lambda}$ we conclude that $X_t^\lambda\leq Y_t^\lambda$, for all $t\in[0,T]$, $\P$-a.s. When $\cO=(0,\infty)$, we simply replace $t$ with $t\wedge\theta_\lambda$ in the argument above. \hfill$\square$

\begin{remark}\label{rem:lemconvex}
When $\cO=(0,\infty)$, if $x_1=0$ and $\nu^1_t\equiv 0$, the assumptions that $\mu(0)=0$ and $\sigma(x)=\sigma_1 x$, imply $X^{\nu^1;x_1}_t\equiv 0$ for all $t\in[0,T]$. Analogously, for $x_2=0$ and $\nu^2_t\equiv 0$, we have $X^{\nu^2;x_2}_t\equiv 0$, $t\in[0,T]$.
A careful inspection of the proof above shows that $x_1=0$ and $\nu^1_t\equiv 0$ imply $X^\lambda_{t\wedge \theta_\lambda}\le (1-\lambda)X^{\nu^2;x_2}_{t\wedge\theta_\lambda}$, whereas $x_2=0$ and $\nu^2_t\equiv 0$ imply $X^\lambda_{t\wedge \theta_\lambda}\le \lambda X^{\nu^1;x_1}_{t\wedge\theta_\lambda}$.
\end{remark}

\subsection{Proof of Lemma \ref{lem:SRP}}
We use a standard approach based on Picard iteration. Full details on this approach for {\em non-reflecting} SDEs can be found, for example, in the proof of \cite[Thm.\ 9.2]{baldi2017stochastic}.
First we assume that $\mu$ and $\sigma$ are Lipschitz and later we relax this assumption to locally Lipschitz.

Let $L>0$ be such that $|\mu(x)-\mu(y)|+|\sigma(x)-\sigma(y)|\leq L|x-y|$ for $x,y\in\cO$. When $\cO=(0,\infty)$ we consider Lipschitz continuous extensions of $\mu$ and $\sigma$ to $\R$ and we denote them again by $\mu$ and $\sigma$. That is required for the Picard iteration procedure which, in each step, produces a process that lives on $\R$. Fix $(t,x)\in[0,T]\times\cO$ and set $X^{(0)}_s\coloneqq x\!-\!(x\!-\!f(t))^+$ and $\nu^{(0)}_s=\!-(x\!-\!f(t))^+$ for $s\in[0,T\!-\!t]$. Define iteratively, for $n\ge 1$, $n\in\N$,
\begin{align*}
X^{(n)}_s\coloneqq &\,x + \int_0^s\!\mu(X^{(n-1)}_u)\, \ud u+ \int_0^s\!\sigma(X^{(n-1)}_u)\, \ud W_u+\nu^{(n)}_s,
\end{align*}
where
\begin{align*}
&\nu^{(n)}_s\coloneqq-\sup_{0\leq u\leq s}\Big(x + \int_0^u\!\mu(X^{(n-1)}_v)\, \ud v+ \int_0^u\!\sigma(X^{(n-1)}_v)\, \ud W_v-f(t+u)\Big)^+.
\end{align*}
It will be shown that the pair $(X^{(n)},\nu^{(n)})$ solves a Skorokhod reflection problem at $f$. Notice though that $X^{(n)}$ does not solve a reflecting SDE.

It is clear by construction that for each $n$ the pair $(X^{(n)},\nu^{(n)})$ is $\F$-adapted, $X^{(n)}_{0}=f(t)\wedge x$, $X^{(n)}_s\in\R$ for all $s\in[0,T-t]$, $\nu^{(n)}$ is non-increasing, with $\nu^{(n)}_0=-(x-f(t))^+$. Possible jumps in the process $\nu^{(n)}$ may only occur at times $s\in[0,T-t]$ when $f(t+s)<\lim_{\eps\downarrow 0}f(t+s+\eps)$. Next we show that indeed $\nu^{(n)}$ is continuous. Notice that 
\begin{align*}
\lim_{\eps\downarrow 0}\nu^{(n)}_{s+\eps}&=-\lim_{\eps\downarrow 0}\sup_{0\leq u\leq s+\eps}\Big(x + \int_0^u\!\mu(X^{(n-1)}_v)\, \ud v+ \int_0^u\!\sigma(X^{(n-1)}_v)\, \ud W_v-f(t+u)\Big)^+\\
&\ge-\lim_{\eps\downarrow 0}\sup_{0\leq u\leq s+\eps}\Big(x + \int_0^u\!\mu(X^{(n-1)}_v)\, \ud v+ \int_0^u\!\sigma(X^{(n-1)}_v)\, \ud W_v-f(t+u\wedge s)\Big)^+\\
&=\nu^{(n)}_s\wedge\Big[-\lim_{\eps\downarrow 0}\sup_{s<u\leq s+\eps}\Big(x + \int_0^u\!\mu(X^{(n-1)}_v)\, \ud v+ \int_0^u\!\sigma(X^{(n-1)}_v)\, \ud W_v-f(t+s)\Big)^+\Big]=\nu^{(n)}_s,
\end{align*}
where the inequality holds because $f$ is non-decreasing and the final equality uses continuity of the process inside the supremum. Since $\nu^{(n)}$ is non-increasing by definition, then the inequality above shows that $\nu^{(n)}$ is right-continuous. The process
\[
u\mapsto x + \int_0^u\!\mu(X^{(n-1)}_v)\, \ud v+ \int_0^u\!\sigma(X^{(n-1)}_v)\, \ud W_v-f(t+u),
\]
is left-continuous, thus left-continuity of $\nu^{(n)}$ follows. Hence $\nu^{(n)}$ and $X^{(n)}$ are continuous on $(0,T-t]$ as claimed. By construction we have $X^{(n)}_s\le f(t+s)$ for all $s\in[0,T-t]$ and it remains to show the Skorokhod condition analogue of the second equation in \eqref{eq:refcond}. Fix $\omega\in\Omega$ and let $s_0\in[0,T-t)$ be such that $X^{(n)}_{s_0}(\omega)<f(t+s_0)$. Then, using the expression of $X^{(n)}$ and rearranging terms we obtain
\begin{align}\label{eq:Sk}
x + \int_0^{s_0}\!\mu\big(X^{(n-1)}_u(\omega)\big)\, \ud u+ \Big(\int_0^{s_0}\!\sigma(X^{(n-1)}_u)\, \ud W_u\Big)(\omega)-f(t+s_0)<-\nu^{(n)}_{s_0}(\omega).
\end{align}
The strict inequality implies that there is $\hat s=\hat s_\omega\in[0,s_0)$ such that $\nu^{(n)}_{u}(\omega)=\nu^{(n)}_{\hat s}(\omega)$ for $u\in(\hat s,s_0]$. Moreover, taking $s_1>s_0$ but sufficiently close to $s_0$, since $f$ is non-decreasing, we have for $s\in(s_0,s_1)$ 
\begin{align*}
&x + \int_0^{s}\!\mu\big(X^{(n-1)}_u(\omega)\big)\, \ud u+ \Big(\int_0^{s}\!\sigma(X^{(n-1)}_u)\, \ud W_u\Big)(\omega)-f(t+s)\\
&\le x + \int_0^{s}\!\mu\big(X^{(n-1)}_u(\omega)\big)\, \ud u+ \Big(\int_0^{s}\!\sigma(X^{(n-1)}_u)\, \ud W_u\Big)(\omega)-f(t+s_0) \le -\nu^{(n)}_{s_0}(\omega),
\end{align*}
where the last inequality holds by continuity.
Then $\nu^{(n)}_{u}(\omega)=\nu^{(n)}_{\hat s}(\omega)$ for $u\in(\hat s,s_1)$ and $\ud \nu^{(n)}_{s_0}(\omega)=0$, as needed. 

Next we want to show that $(X^{(n)},\nu^{(n)})\to (X^{\nu^f\!;x},\nu^f)$ as $n\to\infty$ where $(X^{\nu^f\!;x},\nu^f)$ is the solution of {\bf RSDE}-$f$. With no loss of generality we choose $t=0$. For the first two terms in the Picard's iteration we have, for any $\gamma\in[0,T]$, 
\begin{align}\label{eq:skorgrowbnd}
\begin{aligned}
\E&\Big[\sup_{0\leq s\leq \gamma}\Big|X^{(1)}_s-X^{(0)}_s|^2\Big]\\
\leq &\,3\E\Big[\sup_{0\leq s\leq \gamma}\Big|\int_0^s\!\mu(X^{(0)}_v)\, \ud v\Big|^2+\sup_{0\leq s\leq \gamma}\Big|\int_0^s\!\sigma(X^{(0)}_v)\, \ud W_v\Big|^2+\sup_{0\leq s\leq \gamma}\Big|\nu^{(1)}_s-\nu^{(0)}_s\Big|^2\Big].
\end{aligned}
\end{align}
By the Lipschitz condition on the coefficients we have also that they grow at most linearly with constant which we denote again by $L$, with a slight abuse of notation. Then using that $X^{(0)}_v=x\wedge f(0)$ for $v>0$ and Doob's inequality yields
\begin{align}\label{eq:pic0}
&\E\Big[\sup_{0\leq s\leq \gamma}\Big|\int_0^s\!\mu(X^{(0)}_v)\, \ud v\Big|^2+\sup_{0\leq s\leq \gamma}\Big|\int_0^s\!\sigma(X^{(0)}_v)\, \ud W_v\Big|^2\Big]\le 2L^2 \gamma(4+T)\big[1+(x\wedge f(0))^2\big].
\end{align}
For the remaining term in \eqref{eq:skorgrowbnd} we observe that 
\begin{align*}
\begin{aligned}
0&\le \nu^{(0)}_s-\nu^{(1)}_s=\sup_{0\le v\le s}\Big(x+\mu(x\wedge f(0))v+\sigma(x\wedge f(0))W_v-f(v) \Big)^+-(x-f(0))^+\\
&\le \sup_{0\le v\le s}\Big(x+\mu(x\wedge f(0))v+\sigma(x\wedge f(0))W_v-f(0) \Big)^+-(x-f(0))^+\\
&\le \sup_{0\le v\le s}\Big|\mu(x\wedge f(0))v+\sigma(x\wedge f(0))W_v\Big|,
\end{aligned}
\end{align*}
where in the second inequality we use monotonicity of $f$ and in the final one $(a)^+-(b)^+\le |a-b|$. Then, applying the same bounds as in \eqref{eq:pic0} and combining all the estimates, we obtain from \eqref{eq:skorgrowbnd} 
\begin{align}\label{eq:Cg}
\begin{aligned}
\E\Big[\sup_{0\leq s\leq \gamma}\Big|X^{(1)}_s-X^{(0)}_s|^2\Big]
\leq 18 L^2 \gamma(4+T)\big[1+(x\wedge f(0))^2\big]\eqqcolon C\gamma.
\end{aligned}
\end{align}

We continue our estimates by induction. Assume that for some $n\in\N$, $n\ge 1$
\begin{align*}
\E\Big[\sup_{0\leq t\leq \gamma}\Big|X^{(n)}_t-X^{(n-1)}_t\Big|^2\Big]\leq \frac{(C\gamma)^{n}}{n!}.
\end{align*}
Using the definitions of $X^{(n)}$ and $X^{(n+1)}$ we have
\begin{align}\label{eq:pic1}
\begin{aligned}
&\E\Big[\sup_{0\leq s\leq \gamma}\Big|X^{(n+1)}_s\!-\!X^{(n)}_s\Big|^2\Big]\leq 3\E\Big[\sup_{0\leq s\leq \gamma}\Big|\int_0^s\!\!\big(\mu(X^{(n)}_v)\!-\!\mu(X^{(n-1)}_v)\big)\ud v\Big|^2\Big]\\
&\qquad+3\E\Big[\sup_{0\leq s\leq \gamma}\Big(\Big|\int_0^s\!\!\big(\sigma(X^{(n)}_v)\!-\!\sigma(X^{(n-1)}_v)\big)\ud W_v\Big|^2\!+\!\Big|\nu^{(n+1)}_s\!-\!\nu^{(n)}_s\Big|^2\Big)\Big].
\end{aligned}
\end{align}
Lipschitz continuity of $\mu$ yields
\begin{align}\label{eq:diffmun}
\begin{aligned}
&\E\Big[\sup_{0\leq s\leq \gamma}\Big|\int_0^s\!\big(\mu(X^{(n)}_v)-\mu(X^{(n-1)}_v)\big)\, \ud v\Big|^2\Big]\\
&\leq\E\Big[T\int_0^\gamma\!\sup_{0\leq v\leq s}\big|\mu(X^{(n)}_v)-\mu(X^{(n-1)}_v)\big|^2\, \ud s\Big]\\
&\le T L^2 \int_0^\gamma\!\E\Big[\sup_{0\leq v\leq s}\big|X^{(n)}_v-X^{(n-1)}_v\big|^2\Big]\, \ud s\leq TL^2\int_0^\gamma\!\frac{(C s)^n}{n!}\, \ud s = TL^2C^n\frac{\gamma^{n+1}}{(n+1)!},
\end{aligned}
\end{align}
where we used H\"older's inequality and the inductive hypothesis. Similarly, Lipschitz continuity of $\sigma$, combined with Doob's inequality yields 
\begin{align}\label{eq:diffsign}
&\E\Big[\sup_{0\leq s\leq \gamma}\Big|\int_0^s\!\big(\sigma(X^{(n)}_v)-\sigma(X^{(n-1)}_v)\big)\, \ud W_v\Big|^2\Big]\\
&\leq\E\Big[4\int_0^\gamma\!\sup_{0\leq v\leq s}\big|\sigma(X^{(n)}_v)-\sigma(X^{(n-1)}_v)\big|^2\, \ud s\Big]
\leq 4\int_0^\gamma\! L^2C^n\frac{s^n}{n!}\,\ud s =4L^2C^n \frac{\gamma^{n+1}}{(n+1)!}.
\end{align}

For the third term in \eqref{eq:pic1} we have 
\begin{align}\label{eq:seqcontr}
\begin{aligned}
&\E\Big[\sup_{0\leq s\leq \gamma}\Big|\nu^{(n+1)}_s-\nu^{(n)}_s\Big|^2\Big]\\
&=\E\Big[\sup_{0\leq s\leq \gamma}\Big|\sup_{0\leq v\leq s}\Big(x + \int_0^v\!\mu(X^{(n)}_u)\, \ud u+ \int_0^v\!\sigma(X^{(n)}_u)\, \ud W_u-f(v)\Big)^+\\
&\qquad\qquad\qquad-\sup_{0\leq v\leq s}\Big(x + \int_0^v\!\mu(X^{(n-1)}_u)\, \ud u+ \int_0^v\!\sigma(X^{(n-1)}_u)\, \ud W_u-f(v)\Big)^+\Big|^2\Big]\\
&\leq\E\Big[ \sup_{0\leq s\leq \gamma}\Big|\int_0^s\!\big(\mu(X^{(n)}_u)-\mu(X^{(n-1)}_u)\big) \ud u+ \int_0^s\!\big(\sigma(X^{(n)}_u)-\sigma(X^{(n-1)}_u)\big) \ud W_u\Big|^2\Big]\\
&\leq 2(TL^2+4L^2)C^n\frac{\gamma^{n+1}}{(n+1)!},
\end{aligned}
\end{align}
where for the last inequality we used the same bounds as in \eqref{eq:diffmun} and \eqref{eq:diffsign}. 

Combining the estimates above we obtain 
\begin{align*}
\E\Big[\sup_{0\leq s\leq \gamma}\Big|X^{(n+1)}_s-X^{(n)}_s\Big|^2\Big]\leq\,12\big( TL^2+4L^2\big)C^n\frac{\gamma^{n+1}}{(n+1)!}\leq \frac{(C\gamma)^{n+1}}{(n+1)!},
\end{align*}
by recalling the definition of $C$ in \eqref{eq:Cg}.
It follows that $(X^{(n)})_{n\in\N}$ is a Cauchy Sequence in $L^2(\Omega;L^\infty(0,T))$ and we denote by $X^{\infty}$ its limit. Using the same estimates as above, based on Lipschitz continuity of the coefficients and Doob's inequality we obtain
\begin{align}\label{eq:musigmanto0}
\begin{aligned}
&\E\Big[\sup_{0\leq s\leq T}\Big|\int_0^s\!\big(\mu(X^{(n)}_v)-\mu(X^{\infty}_v)\big) \ud v\Big|^2\Big]\leq T L^2\E\Big[\sup_{0\leq s\leq T}|X^{(n)}_s-X^{\infty}_s|^2\Big]\xrightarrow{n\to\infty} 0, \\
&\E\Big[\sup_{0\leq s\leq T}\Big|\int_0^s\!\big(\sigma(X^{(n)}_v)-\sigma(X^{\infty}_v)\big) \ud W_v\Big|^2\Big]\leq 4TL^2\E\Big[\sup_{0\leq s\leq T}|X^{(n)}_s-X^{\infty}_s|^2\Big]\xrightarrow{n\to\infty} 0.
\end{aligned}
\end{align}
Define
\begin{align}\label{eq:optcontrde}
\nu_s^{\infty}=-\sup_{0\leq v\leq s}\Big(x + \int_0^v\!\mu(X^{\infty}_u)\, \ud u+ \int_0^v\!\sigma(X^{\infty}_u)\, \ud W_u-f(v)\Big)^+.
\end{align}
Again, arguing as in \eqref{eq:seqcontr} we obtain 
\begin{align}\label{eq:convnu}
\begin{aligned}
&\E\Big[\sup_{0\leq s\leq T}|\nu_s^{(n)}\!-\!\nu_s^{\infty}|^2\Big]\\
&\leq\E\Big[\sup_{0\leq s\leq T}\Big|\int_0^s\!\big(\mu(X^{(n)}_u)\!-\!\mu(X^{\infty}_u)\big) \ud u\!+\! \int_0^s\!\big(\sigma(X^{(n)}_u)\!-\!\sigma(X^{\infty}_u)\big) \ud W_u\Big|^2\Big],
\end{aligned}
\end{align}
hence implying that $\nu^{(n)}\to \nu^{\infty}$ in $L^2(\Omega;L^{\infty}(0,T))$, by \eqref{eq:musigmanto0}. Putting together \eqref{eq:musigmanto0} and \eqref{eq:convnu} we have that $(X^{\infty},\nu^\infty)$ solves \eqref{eq:RSDE}. It remains to show it is also solution of {\bf RSDE}-$f$. 

It is clear that $\nu^\infty_0=-(x-f(0))^+$. Continuity of $t\mapsto \nu^\infty_t$ for $t\in[0,T]$ is obtained by the same argument as for the continuity of $\nu^{(n)}$. That also implies continuity of $X^\infty$ and $X^\infty_0=x\wedge f(0)$. The condition $X^\infty_s\le f(s)$ for all $s\in[0,T]$ follows by direct inspection of \eqref{eq:RSDE}. The second condition of \eqref{eq:refcond} holds by a repetition of the arguments around \eqref{eq:Sk}, with $(X^{(n-1)},\nu^{(n)})$ therein replaced by $(X^\infty,\nu^\infty)$. Finally, for $\cO=\R$ it is clear that $X^\infty\in\R$ (i.e., it is always finite) because $X^\infty\in L^2(\Omega;L^\infty(0,T))$. Therefore $\nu^\infty\in\cA_{0,x}$. Instead, for $\cO=(0,\infty)$ it suffices to notice that the process $X^\infty$ is uncontrolled away from the boundary $f$ thanks to the second condition in \eqref{eq:refcond}. Therefore the boundary point $\{0\}$ is not attainable, as for the process $X^0$ defined in \eqref{eq:SDE0}, and $\nu^\infty\in\cA_{0,x}$.
Thus, we have shown that $(X^\infty,\nu^\infty)$ yields a solution to {\bf RSDE}-$f$ under the assumption of Lipschitz coefficients $\mu$ and $\sigma$.

Uniqueness is proved using an approach that can be found, e.g., in \cite[Thm.\ 1.2.1]{pilipenko2014reflection}. Let $(X^{\nu^1},\nu^1)$ and $(X^{\nu^2},\nu^2)$ be two solution pairs of {\bf RSDE}-$f$. Applying It\^o's formula we have
\begin{align}\label{eq:SDErefuniq}
\begin{aligned}
(X^{\nu^1}_s-X^{\nu^2}_s)^2=&\,\int_0^s\!\Big[2(X^{\nu^1}_u-X^{\nu^2}_u)\big(\mu(X^{\nu^1}_u)-\mu(X^{\nu^2}_u)\big)+\big(\sigma(X_u^{\nu^1})-\sigma(X^{\nu^2}_u)\big)^2\Big]\,\ud u\\
&\,+2\int_0^s\!(X^{\nu^1}_u-X^{\nu^2}_u)\,(\ud \nu^1_u-\ud \nu^2_u)\\
&\,+2\int_0^s\!(X^{\nu^1}_u-X^{\nu^2}_u)\big(\sigma(X^{\nu^1}_u)-\sigma(X^{\nu^2}_u)\big)\,\ud W_u.
\end{aligned}
\end{align}
By the Skorokhod condition, the second integral on the right-hand side above can be rewritten as
\begin{align*}
&\int_0^s\!(X^{\nu^1}_u\!-\!X^{\nu^2}_u)(\ud \nu^1_u\!-\!\ud \nu^2_u)\\
&=\int_0^s\!\mathds{1}_{\{X^{\nu^1}_{u}=f(u)\}}(X^{\nu^1}_u\!-\!X^{\nu^2}_u)\,\ud \nu^1_u\!-\!\int_0^s\!\mathds{1}_{\{X^{\nu^2}_{u}=f(u)\}}(X^{\nu^1}_u\!-\!X^{\nu^2}_u)\,\ud \nu^2_u\leq 0,
\end{align*}
where the inequality holds because $\nu^1$ and $\nu^2$ are non-increasing and $\max\{X^{\nu^1}_u,X^{\nu^2}_u\}\le f(u)$ for all $u\in[0,s]$.
Thus, taking expectation in \eqref{eq:SDErefuniq}, we obtain 
\begin{align*}
\E\Big[|X^{\nu^1}_s-X^{\nu^2}_s|^2\Big]\leq \E\Big[(2L+L^2)\int_0^s|X^{\nu^1}_u-X^{\nu^2}_u|^2 \,\ud u\Big],
\end{align*}
by the Lipschitz property of $\mu$ and $\sigma$ and upon noticing that the stochastic integral is a (local) martingale, so that we can remove it (possibly using a standard localisation procedure). An application of Gronwall's lemma yields $\P(X^{\nu^1}_s=X^{\nu^2}_s)=1$ for all $s\in[0,T]$ and, by continuity of the trajectories, $\P(X^{\nu^1}_s=X^{\nu^2}_s,\,s\in[0,T])=1$ as usual. Hence, we also have $\P(\nu^1_s=\nu^2_s,\,s\in[0,T])=1$. Thus uniqueness is proven when $\mu$ and $\sigma$ are Lipschitz.

It remains to relax the Lipschitz assumption on $\mu$ and $\sigma$. For the ease of exposition but with no loss of generality here we consider $\cO=\R$. From now on we assume $\mu$ and $\sigma$ locally Lipschitz with linear growth. Following a classical approach (e.g., see \cite[Thm.\ 9.4]{baldi2017stochastic}). Let $(\phi_j)_{j\in\N}$ be $C^{\infty}$-functions on $\R$ such that $0\le \phi_j\le 1$, $\phi_j(x)=1$ for $|x|\leq j$ and $\phi_j(x)=0$ for $|x|\geq j+1$. Let $\mu_j(x)=\mu(x)\phi_j(x)$ and $\sigma_j(x)=\sigma(x)\phi_j(x)$. From the proof above, there exists a unique pair $(X^{\nu^j},\nu^j)$ that characterises a solution of {\bf RSDE}-$f$ with coefficients $\mu_j$ and $\sigma_j$. 

Set $\tau_j\coloneqq\inf\{s\ge 0: X^{\nu^j}\notin(-j,j)\}$. Thanks to uniqueness, for any $k\in\N$, $k\ge j$, we also have $\tau_j=\inf\{s\ge 0: X^{\nu^k}\notin(-j,j)\}$ and, more importantly, 
\[
(X^{\nu^j}_{s\wedge\tau_j},\nu^j_{s\wedge\tau_j})_{s\in[0,T]}=(X^{\nu^k}_{s\wedge\tau_j},\nu^k_{s\wedge\tau_j})_{s\in[0,T]}
\] 
up to indistinguishability. This observation uniquely defines a pair $(X^{\nu^f},\nu^f)$ that solves {\bf RSDE}-$f$ with coefficients $\mu$ and $\sigma$, on the stochastic interval $[\![0,\tau_j]\!]$, because $(\mu_j,\sigma_j)=(\mu,\sigma)$ on the interval $(-j,j)$. 
Since $\tau_j\le \tau_{j+1}$ then we can define $\tau_\infty=\lim_{j\to\infty}\tau_j$ and we want to show that $\tau_\infty>T$, $\P$-a.s. Then the pair $(X^{\nu^f},\nu^f)$ solves {\bf RSDE}-$f$ uniquely on $[0,T]$. 

Since
\begin{align}
\begin{aligned}
\big|X^{\nu^j}_s\big|^2\le 4 \Big(x^2+ \Big|\int_0^{s}\mu(X^{\nu^j}_u)\ud u\Big|^2+\Big|\int_0^{s}\sigma(X^{\nu^j}_u)\ud W_u\Big|^2+|\nu^j_s|^2\Big),
\end{aligned}
\end{align}
then linear growth of $\mu$ and $\sigma$ and 
\begin{align}\label{eq:nuj}
\begin{aligned}
|\nu^j_s|=&\,\sup_{0\leq u\leq s}\Big(x + \int_0^u\!\mu_j(X^{\nu^j}_v)\, \ud v+ \int_0^u\!\sigma_j(X^{\nu^j}_v)\, \ud W_v-f(u)\Big)^+\\
\leq &\,|f(0)|+\sup_{0\leq u\leq s}\Big|x + \int_0^u\!\mu_j(X^{\nu^j}_v)\, \ud v+ \int_0^u\!\sigma_j(X^{\nu^j}_v)\, \ud W_v\Big|,
\end{aligned}
\end{align}
allow us to repeat usual estimates based on Doob's inequality (e.g., as in \eqref{eq:skorgrowbnd}) and Gronwall's lemma. Then we obtain 
\begin{align}\label{eq:L2Xnu}
\sup_{j\in\N}\E\Big[\sup_{0\leq s\leq T}|X^{\nu^j}_s|^2\Big]\leq C,
\end{align}
for some $C>0$ independent of $j\in\N$. Finally, 
\begin{align*}
\P(\tau_j\leq T)\leq \P\Big(\sup_{0\leq s\leq T}|X^{\nu^j}_s|\geq j\Big)\leq \frac{1}{j^2}\E\Big[\sup_{0\leq s\leq T}|X^{\nu^j}_s|^2\Big]\leq\frac{C}{j^2}\xrightarrow{j\to\infty}0.
\end{align*}
That shows $\P(\tau_\infty>T)=1$ as needed. It then follows from \eqref{eq:L2Xnu} that $\E[\sup_{0\leq s\leq T}|X^{\nu^f}_s|^2]\leq C$. That, combined with \eqref{eq:nuj}, allows us to conclude that \eqref{eq:nuL2} also holds by linear growth of $\mu$ and $\sigma$, using again estimates as above based on H\"older's inequality and Doob's inequality.\hfill$\square$

\medskip
{\bf Acknowledgments}: This work was presented in June 2023 at the 11th General AMaMeF conference in Bielefeld. We thank participants of the conference for useful feedback. Both authors received partial financial support from EU -- Next Generation EU -- PRIN2022 (2022BEMMLZ) CUP: D53D23005780006. T.\ De Angelis also received partial financial support from PRIN-PNRR2022 (P20224TM7Z) CUP: D53D23018780001.

\bibliographystyle{plain}
\bibliography{Bibliography}

\begin{thebibliography}{10}

\bibitem{baldi2017stochastic}
P.~Baldi.
\newblock {\em Stochastic calculus}.
\newblock Universitext. Springer, 2017.

\bibitem{bayraktar2013controller}
E.~Bayraktar and Y.-J. Huang.
\newblock On the multidimensional controller-and-stopper games.
\newblock {\em {SIAM} J.\ Control Optim.}, 51(2):1263--1297, 2013.

\bibitem{bensoussan1974nonlinear}
A.~Bensoussan and A.~Friedman.
\newblock Nonlinear variational inequalities and differential games with
  stopping times.
\newblock {\em J.\ Funct.\ Anal.}, 16(3):305--352, 1974.

\bibitem{bensoussan1977nonzero}
A.~Bensoussan and A.~Friedman.
\newblock Nonzero-sum stochastic differential games with stopping times and
  free boundary problems.
\newblock {\em Trans.\ Amer.\ Math.\ Soc.}, 231(2):275--327, 1977.

\bibitem{boetius1998connection}
F.~Boetius and M.~Kohlmann.
\newblock Connections between optimal stopping and singular stochastic control.
\newblock {\em Stochastic Process.\ Appl.}, 77(2):253--281, 1998.

\bibitem{bovo2023c}
A.~Bovo and T.~De~Angelis.
\newblock Finite-time horizon, stopper vs.\ singular-controller games on the
  half-line.
\newblock {\em ar{X}iv:2409.06049}, 2024.

\bibitem{bovo2022variational}
A.~Bovo, T.~De~Angelis, and E.~Issoglio.
\newblock Variational inequalities on unbounded domains for zero-sum
  singular-controller vs. stopper games.
\newblock {\em To appear in Math.\ Oper.\ Res. (arXiv:2203.06247)}, 2022.

\bibitem{bovo2023b}
A.~Bovo, T.~De~Angelis, and J.~Palczewski.
\newblock Stopper vs.\ singular-controller games with degenerate diffusions.
\newblock {\em arXiv:2312.00613}, 2023.

\bibitem{bovo2023}
A.~Bovo, T.~De~Angelis, and J.~Palczewski.
\newblock Zero-sum stopper vs. singular-controller games with constrained
  control directions.
\newblock {\em {SIAM} J.\ Control Optim.}, 62(4):2203--2228, 2024.

\bibitem{burdzy2009skorokhod}
K.~Burdzy, W.~Kang, and K.~Ramanan.
\newblock The {S}korokhod problem in a time-dependent interval.
\newblock {\em Stochastic Process.\ Appl.}, 119(2):428--452, 2009.

\bibitem{chiarolla2000inflation}
M.B. Chiarolla and U.G. Haussmann.
\newblock Controlling inflation: The infinite horizon case.
\newblock {\em Appl.\ Math.\ Optim.}, 41:25--50, 2000.

\bibitem{chiarolla2009irreversible}
M.B. Chiarolla and U.G. Haussmann.
\newblock On a stochastic, irreversible investment problem.
\newblock {\em SIAM J.\ Control Optim.}, 48(2):438--462, 2009.

\bibitem{chow1985additive}
P.-L. Chow, J.-L. Menaldi, and M.~Robin.
\newblock Additive control of stochastic linear systems with finite horizon.
\newblock {\em SIAM J.\ Control Optim.}, 23(6):858--899, 1985.

\bibitem{de2015note}
T.~De~Angelis.
\newblock A note on the continuity of free-boundaries in finite-horizon optimal
  stopping problems for one-dimensional diffusions.
\newblock {\em SIAM J.\ Control Optim.}, 53(1):167--184, 2015.

\bibitem{de2017dividend}
T.~De~Angelis and E.~Ekstr{\"o}m.
\newblock The dividend problem with a finite horizon.
\newblock {\em Ann.\ Appl.\ Probab.}, 27(6):3525--3546, 2017.

\bibitem{dellacherie1982probabilities}
C.~Dellacherie and P.-A. Meyer.
\newblock {\em Probabilities and potential B, volume 72 of North-Holland
  Mathematics Studies}.
\newblock North-Holland Publishing Co., Amsterdam, 1982.

\bibitem{ekstrom2020detect}
E.~Ekstr\"om, K.~Lindensj\"o, and M.~Olofsson.
\newblock How to detect a salami slicer: a stochastic controller-and-stopper
  game with unknown competition.
\newblock {\em {SIAM} J.\ Control Optim.}, 60(1):545--574, 2022.

\bibitem{ekstrom2023finetti}
E.~Ekstr{\"o}m, A.~Milazzo, and M.~Olofsson.
\newblock The de {F}inetti problem with uncertain competition.
\newblock {\em SIAM J.\ Control Optim.}, 61(5):2997--3017, 2023.

\bibitem{karoui1991new}
N.~El~Karoui and I.~Karatzas.
\newblock A new approach to the {S}korohod problem, and its applications.
\newblock {\em Stoch.\ Stoch.\ Rep.}, 34(1-2):57--82, 1991.

\bibitem{evans1979second}
L.C. Evans.
\newblock A second order elliptic equation with gradient constraint.
\newblock {\em Comm.\ Partial Differential Equations}, 4(5):555--572, 1979.

\bibitem{fleming2012deterministic}
W.H. Fleming and R.W. Rishel.
\newblock {\em Deterministic and stochastic optimal control}, volume~1 of {\em
  Applications of {M}athematics}.
\newblock Springer-Verlag, Berlin-New York, 1975.

\bibitem{fleming2006controlled}
W.H. Fleming and H.M. Soner.
\newblock {\em Controlled {M}arkov processes and viscosity solutions},
  volume~25 of {\em Stochastic Modelling and Applied Probability}.
\newblock Springer New York, second edition, 2006.

\bibitem{friedman2008partial}
A.~Friedman.
\newblock {\em Partial differential equations of parabolic type}.
\newblock Dover Publications, Mineola, New York, 2008.

\bibitem{hamadene2006stochastic}
S.~Hamadène.
\newblock Mixed zero-sum stochastic differential game and {American} game
  options.
\newblock {\em SIAM J.\ Control Optim.}, 45(2):496--518, 2006.

\bibitem{haussmann1995singular}
U.G. Haussmann and W.~Suo.
\newblock Singular optimal stochastic controls {I}: Existence.
\newblock {\em SIAM J.\ Control Optim.}, 33(3):916--936, 1995.

\bibitem{hernandez2015zero}
D.~Hernandez-Hernandez, R.S. Simon, and M.~Zervos.
\newblock A zero-sum game between a singular stochastic controller and a
  discretionary stopper.
\newblock {\em Ann.\ Appl.\ Probab.}, 25(1):46--80, 2015.

\bibitem{hernandez2015zsgsingular}
D.~Hernandez-Hernandez and K.~Yamazaki.
\newblock Games of singular control and stopping driven by spectrally one-sided
  {L}évy processes.
\newblock {\em Stochastic Process.\ Appl.}, 125(1):1--38, 2015.

\bibitem{jeanblanc1995optimization}
M.~Jeanblanc-Picqu{\'e} and A.N. Shiryaev.
\newblock Optimization of the flow of dividends.
\newblock {\em Uspekhi Mat.\ Nauk}, 50(2):25--46, 1995.

\bibitem{karatzas1984bridge1}
I.~Karatzas and S.E. Shreve.
\newblock Connections between optimal stopping and singular stochastic control
  {I}. {M}onotone follower problems.
\newblock {\em {SIAM} J.\ Control Optim.}, 22(6):856--877, 1984.

\bibitem{karatzas1985bridge2}
I.~Karatzas and S.E. Shreve.
\newblock Connections between optimal stopping and singular stochastic control
  {II}. {R}eflected follower problems.
\newblock {\em {SIAM} J.\ Control Optim.}, 23(3):433--451, 1985.

\bibitem{karatzas2001controller}
I.~Karatzas and W.D. Sudderth.
\newblock The controller-and-stopper game for a linear diffusion.
\newblock {\em Ann.\ Probab.}, 29(3):1111--1127, 2001.

\bibitem{karatzas2008martingale}
I.~Karatzas and I.-M. Zamfirescu.
\newblock Martingale approach to stochastic differential games of control and
  stopping.
\newblock {\em Ann.\ Probab.}, 36(4):1495--1527, 2008.

\bibitem{krylov2008lectures}
N.V. Krylov.
\newblock {\em Lectures on elliptic and parabolic equations in {S}obolev
  spaces}, volume~96 of {\em Graduate Studies in Mathematics}.
\newblock American Mathematical Society, Providence, RI, 2008.

\bibitem{krylov1980controlled}
N.V. Krylov.
\newblock {\em Controlled diffusion processes}, volume~14 of {\em Stochastic
  Modelling and Applied Probability}.
\newblock Springer Berlin, 2009.
\newblock Translated from the 1977 Russian original by A.B. Aries, Reprint of
  the 1980 edition.

\bibitem{lamberton2008critical}
D.~Lamberton and M.~Mikou.
\newblock The critical price for the {A}merican put in an exponential
  {L}{\'e}vy model.
\newblock {\em Finance Stoch.}, 12(4):561--581, 2008.

\bibitem{lieberman1996second}
G.M. Lieberman.
\newblock {\em Second order parabolic differential equations}.
\newblock World Scientific Publishing Co., Inc., River Edge, NJ, 1996.

\bibitem{maitra1996gambler}
A.P. Maitra and W.D. Sudderth.
\newblock The gambler and the stopper.
\newblock In {\em Statistics, probability and game theory}, volume~30 of {\em
  IMS Lecture Notes Monogr.\ Ser.}, pages 191--208. Inst.\ Math.\ Statist.,
  Hayward, CA, 1996.

\bibitem{menaldi1989optimal}
J.L. Menaldi and M.I. Taksar.
\newblock Optimal correction problem of a multidimensional stochastic system.
\newblock {\em Automatica J.\ IFAC}, 25(2):223--232, 1989.

\bibitem{pilipenko2014reflection}
A.~Pilipenko.
\newblock {\em An introduction to stochastic differential equations with
  reflection}, volume~1 of {\em Lectures in {P}ure and {A}pplied
  {M}athematics}.
\newblock Universit\"atsverlag Potsdam, 2014.

\bibitem{protter2005stochastic}
P.E. Protter.
\newblock {\em Stochastic integration and differential equations}, volume~21 of
  {\em Stochastic Modelling and Applied Probability}.
\newblock Springer Berlin, second edition, 2005.

\bibitem{radner1996risk}
R.~Radner and L.~Shepp.
\newblock Risk vs.\ profit potential: {A} model for corporate strategy.
\newblock {\em J.\ {E}con.\ Dyn.\ Control}, 20(8):1373--1393, 1996.

\bibitem{soner1989regularity}
H.M. Soner and S.E. Shreve.
\newblock Regularity of the value function for a two-dimensional singular
  stochastic control problem.
\newblock {\em {SIAM} J.\ Control Optim.}, 27(4):876--907, 1989.

\bibitem{soner1991free}
H.M. Soner and S.E. Shreve.
\newblock A free boundary problem related to singular stochastic control: {t}he
  parabolic case.
\newblock {\em Comm.\ Partial Differential Equations}, 16(2-3):373--424, 1991.

\end{thebibliography}
\end{document}